\documentclass[moor,sglanonrev]{informs4}
\usepackage{eqndefns-left} 
\RequirePackage{tgtermes}
\RequirePackage{newtxtext}
\RequirePackage{newtxmath}
\RequirePackage{bm}
\RequirePackage{endnotes}

\usepackage{fancyhdr}
\makeatletter
\long\def\theARTICLETOPLEFT{}
\long\def\theARTICLETOPRIGHT{}
\makeatother

\SingleSpacedXII
\usepackage[colorlinks=true, linkcolor=blue, citecolor=blue, urlcolor=blue]{hyperref}
\usepackage{algorithm}
\usepackage{algpseudocode}
\usepackage{tikz}
\usepackage{enumerate}
\usepackage[sort&compress]{natbib}

\makeatletter
\def\theARTICLEABSTRACT{%
  \HOOKb
  \vspace*{18pt}
  \begin{minipage}[t]{\textwidth}\parindent1em\ABSfont
    \noindent\theABSTRACT\endgraf
    \vskip5pt
    \theFUNDING
    \theKEYWORDS
    \theSUBJECTCLASS
    \theAREAOFREVIEW
    \theORMSCLASS
    \if@BLINDREV\else\theHISTORY\fi
    \noindent\textbf{Citation notice:} This is the author accepted manuscript version of the article published in \emph{Mathematics of Operations Research}. The final version is available at: \url{https://doi.org/10.1287/moor.2024.0609}.
    \vskip5pt
  \noindent\hrulefill
  \end{minipage}
  \vspace*{0pt}
}
\makeatother

 \bibpunct[, ]{[}{]}{,}{n}{}{,}%
 \def\bibfont{\small}%
 \def\e{\varepsilon}
 
 \def\bx{\bar{x}}
 \def\bp{\bar{p}}
 \def\bu{\bar{u}}
 \def\bv{\bar{v}}
 \def\bw{\bar{w}}
 \def\tlam{\tilde{\lambda}}
 \def\blam{\bar{\lambda}}
 \def\hes{\nabla^2}
 \def\gph{{\rm gph}}
 \def\S{S_{\rm KKT}}
 \def\L{L_{\rm KKT}}
 
 \def\tu{\tilde{u}}

 \def\R{\mathbb{R}}

 \def\mypi{\Pi_{\mathbb{R}^{|I_1|}}}
 \def\T{T_{\rm KKT}}
 
 \def\eps{\varepsilon}
\EquationsNumberedBySection 

\TheoremsNumberedBySection  
\ECRepeatTheorems  %

\begin{document}


\RUNAUTHOR{Diao, Dai and Zhang}

\RUNTITLE{Stability for Nash Equilibrium Problems}

\TITLE{Stability for Nash Equilibrium Problems}

\ARTICLEAUTHORS{
\AUTHOR{Ruoyu Diao\textsuperscript{a,b}, Yu-Hong Dai\textsuperscript{a,b}, Liwei Zhang$^{\rm c,d,*}$}
\AFF{\textsuperscript{a}Academy of Mathematics and Systems Science, Chinese Academy of Sciences, 100190 Beijing, China; \textsuperscript{b}School of Mathematical Sciences, University of Chinese Academy of Sciences, 100049 Beijing; \textsuperscript{c}National Frontiers Science Center for Industrial Intelligence and Systems Optimization, Northeastern University, 110819 Shenyang, China; \textsuperscript{d}Key Laboratory of Data Analytics and Optimization for Smart Industry (Northeastern University), Ministry of Education, 110819 Shenyang, China}
\AFF{\textsuperscript{*}Corresponding author}
\AFF{\EMAIL{diaoruoyu18@mails.ucas.ac.cn}; \EMAIL{dyh@lsec.cc.ac.cn}; \EMAIL{Zhanglw@mail.neu.edu.cn}}


} 

\ABSTRACT{%
Consider the stability properties of the Karush-Kuhn-Tucker ({\rm KKT}) solution mapping $S_{\rm KKT}$ for Nash equilibrium problems (NEPs) with canonical perturbations. Firstly, we obtain an exact characterization of the strong regularity of $S_{\rm KKT}$ as well as an easily-verified sufficient condition. Secondly, we propose equivalent conditions for the continuously differentiable single-valued localization of $S_{\rm KKT}$. Thirdly, the isolated calmness of $S_{\rm KKT}$ is studied based on the I-property. The P-property is proposed as a sufficient condition for the robust isolated calmness of $S_{\rm KKT}$ under convex assumptions. Furthermore, we establish that studying the stability properties of the NEP with canonical perturbations is equivalent to studying those of the NEP with only tilt perturbations based on the prior discussions. Finally, we provide  detailed characterizations of stability for the NEPs in which each individual player solves a quadratic programming problem.
}%

\FUNDING{This research was supported by the Strategic Priority Research Program of Chinese
Academy of Sciences (No. XDA27010101). The second author was supported by the Natural Science Foundation of China  (Nos. 11991021, 11991020 and 12021001). The third author was supported by the Major Program of National Natural Science Foundation of China (Nos. 72192830 and 72192831), National Natural Science Foundation of China (No.12371298) and the 111 Project (B16009).}



\KEYWORDS{Nash equilibrium problems, Strong regularity, Continuously differentiable single-valued localization, Robust isolated calmness, Aubin property, Strict Mangasarian-Fromovitz constraint qualification, Quadratic programming} 



\maketitle


\section{Introduction}\label{sec:Intro}

 Nash equilibrium problems (NEPs) are a significant class of game models arising in fields like machine learning (\citet{farnia2020gans}), supply chain management (\citet{contreras2004numerical}) and economics (\citet{myerson2013game}). For comprehensive surveys and general references, see \citet{aubin2013optima, facchinei2010generalized, young2014handbook}.
 In investigating the convergence of algorithms, it is important to assess whether the perturbations in NEPs would impact the solution. Typically, we study stability to ensure the convergence of these algorithms.

 Let $N$ be the number of players in a Nash equilibrium problem, $x^k:=(x^k_1,...,x^k_{n_k})\in \mathbb{R}^{n_k}$ represent the strategy of the $k$-th player,  
 and $x^{-k}:=(x^1,...,x^{k-1},x^{k+1},$ $...,x^{N})$ denote the strategies of the other players.
 In an NEP with canonical perturbations, the $k$-th player's best strategy is the minimizer of the following optimization problem with canonical perturbations
 \begin{equation}
    \textbf{P}^k (x^{-k},p^k)\quad \quad
    \begin{array}{ll}
     \min\limits_{x^k \in \mathbb{R}^{n_k}} & f^k  (x^k, x^{-k}; w^k)-\left< v^k , x^k\right>\\
   {\rm s.\ t.}& g_i^k (x^k; w^k) = u^k_i, \  i=1,...,s_k,\\
   & g^k_j  (x^k; w^k) \le u^k_j,\ j=s_k+1,...,m_k\\
   \end{array}
   \end{equation}
for the given other players' strategies $x^{-k}$. In the above, $w^k \in \mathbb{R}^{d_k}$, $v^k \in \mathbb{R}^{n_k}$ and $ u^k: =(u^k_1,...,u^k_{m_k})\in \mathbb{R}^{m_k}$ are perturbation parameters. The perturbation parameter $p^k:=(u^k,v^k,w^k)$, and $p:=(p^1,...,p^N)$. The $k$-th player has an objective function 
$f^k\in C^2  (\mathbb{R}^{n}\times \mathbb{R}^{d_k})$, and constraints $g^k_i\in C^2 (\mathbb{R}^{n_k}\times \mathbb{R}^{d_k})$ for $i=1,...,m_k$. Let $x:=(x^1,...,x^N)$, $u:=(u^1,...,u^N)$, $v:=(v^1,...,v^N)$, and $w:=(w^1,...,w^N)$. The total dimensions of the strategy $x$ and the perturbation parameter $w$ are $n:=n_1+\cdots+n_N$ and $d:=d_1+\cdots +d_N$, respectively. Then the NEP at the perturbation $p$ can be defined as 
\begin{equation}\label{NEP_perturbations}
    {\rm finding} \ \bar{x}\in \mathbb{R}^n\  {\rm such\  that\  
    each\ } \bar{x}^k\  {\rm is\ a\ minimizer\ of\ } \textbf{P}^k (\bar{x}^{-k},p^k),
    \end{equation}
 where $\bx: = (\bx^1,...,\bx^N)$. A solution of \eqref{NEP_perturbations} is called a Nash equilibrium (NE) at $p$. We define a local Nash equilibrium (LNE) at $p$ as $\bx$, where each $\bx^k$ is a local minimizer of $\textbf{P}^k (\bar{x}^{-k},p^k)$.

The ordinary Lagrange function is defined as
$$
    L^k  (x^k, x^{-k}, \lambda^k ;w^k ): = f^k (x^k,x^{-k}; w^k) + \sum_{i=1}^{s_k} \lambda^k_i g^k_i  (x^k;w^k) + \sum_{j=s_{k}+1}^{m_k}\lambda^k_j g^k_j  (x^k;w^k)
$$
for $k=1,...,N$. Let 
$$
    G^k (x^k;w^k): = 
    \begin{pmatrix}
        g^k_1  (x^k;w^k)\\
        \vdots \\
        g^k_{m_k}  (x^k;w^k)
    \end{pmatrix}, \quad k=1,...,N. 
$$
Let $S_{\rm KKT}: \mathbb{R}^{m+n+d}  \rightrightarrows\  \mathbb{R}^{n+m}$ be the Karush-Kuhn-Tucker (KKT) solution mapping of NEP  (\ref{NEP_perturbations}) such that $\S (p)$ is the solution set of the following generalized equation
\begin{equation}\label{NEP-KKT}
    \left\{\begin{array}{rl}
        v^1 &= \nabla_{x^1} L^1  (x^1,x^{-1},\lambda^1;w^1),\\
        &\vdots\\
        v^N &= \nabla_{x^N} L^N  (x^N,x^{-N},\lambda^N;w^N),\\
        -u^1 &\in -G^1  (x^1; w^1) + N_{\mathbb{R}^{s_1} \times \mathbb{R}_+^{m_1-s_1}} (\lambda^1),\\
        &\vdots\\
        -u^N &\in -G^N(x^N;w^N)+N_{\mathbb{R}^{s_N} \times \mathbb{R}_+^{m_N-s_N}} (\lambda^N),
    \end{array}\right.
\end{equation}
 where $N_{\mathbb{R}^{s_k} \times \mathbb{R}_+^{m_k-s_k}} (\lambda^k)$ is the normal cone to $\mathbb{R}^{s_k} \times \mathbb{R}_+^{m_k-s_k}$ at $\lambda^k$ for $k=1,...,N$ (see \citet[Definition 6.3]{rockafellar2009variational}). The linearization of \eqref{NEP-KKT} at a given point $(\bp,\bx,\blam)$ is the following generalized equation system:
\begin{equation}\label{NEP-LKKT}
    \left\{\begin{array}{rl}
        v^1 &= \nabla_{x^1}L^1(\bar{x}^1,\bar{x}^{-1},\bar{\lambda}^1;\bar{w}^1)+ \sum\limits_{i=1}^{N} \nabla^2_{x^1 x^i}L^1(\bar{x}^1,\bar{x}^{-1},\bar{\lambda}^1;\bar{w}^1)(x^i-\bar{x}^{i})\\
        &+\mathcal{J}_{x^1}G^1(\bar{x}^1;\bar{w}^1)^{\mathrm{T}} (\lambda^1-\bar{\lambda}^1),\\
        &\vdots\\
        v^N &= \nabla_{x^N}L^N(\bar{x}^N,\bar{x}^{-N},\bar{\lambda}^N;\bar{w}^N)+ \sum\limits_{i=1}^{N} \nabla^2_{x^N x^i}L^N(\bar{x}^N,\bar{x}^{-N},\bar{\lambda}^N;\bar{w}^N)(x^i-\bar{x}^{i})\\
        &+\mathcal{J}_{x^N}G^N(\bar{x}^N;\bar{w}^N)^{\mathrm{T}} (\lambda^N-\bar{\lambda}^N),\\
        -u^1 &\in -G^1(\bar{x}^1;\bar{w}^1)-\mathcal{J}_{x^1}G^1(\bar{x}^1;\bar{w}^1)(x^1-\bar{x}^1)+N_{\mathbb{R}^{s_1} \times \mathbb{R}_+^{m_1-s_1}}(\lambda^1),\\
        &\vdots\\
        -u^N &\in -G^N(\bar{x}^N;\bar{w}^N)-\mathcal{J}_{x^N}G^N(\bar{x}^N;\bar{w}^N)(x^N-\bar{x}^N)+N_{\mathbb{R}^{s_N} \times \mathbb{R}_+^{m_N-s_N}}(\lambda^N),\\
    \end{array}\right.
\end{equation}
where $\mathcal{J}_{x^k}G^k(\bar{x}^k; \bar{w}^k)$ denotes the partial Jacobian matrix of $G^k$ at $(\bar{x}^k;\bar{w}^k)$ with respect to $x^k$ for $k=1,...,N$. The set-valued mapping assigning to each $(u,v)$ the solution set of \eqref{NEP-LKKT} is denoted by $\L$. Let $X_{\rm KKT}(p):= \{x(p) \,|\, \exists\, \lambda(p) \ {\rm s.\ t.}\ \left(x(p),\lambda(p)\right)\in S_{\rm KKT}(p)\} $. When $N=1$, $S_{\rm KKT}$ is the KKT solution mapping of a nonlinear programming (NLP) problem  with canonical perturbations. Let 
$
     \T: \mathbb{R}^m \times \mathbb{R}^n \rightarrow \mathbb{R}^n\times \mathbb{R}^m
$
be a set-valued mapping such that
$$
     \T(u,v) = S_{\rm KKT}(u,v,\bar{w}).
$$
Then $\T$ is actually the KKT solution mapping of an NEP with only tilt perturbations (see \citet{rockafellar2024generalized}).

To the best of our knowledge, there is limited research  regarding the stability of a standard NEP. In the literature \citep{kojima1985strongly}, Kojima et al. utilized the classical Kojima's theory (\citet{kojima1978studies}) of locally nonsingular ${\rm PC^1}$ functions to obtain a criterion for the continuous single-valued localization of $S_{\rm KKT}$ (see Definition \ref{def-localization} in Section \ref{Pre}) in multi-person noncooperative games, referred to as ``strong stability''. However, they only considered a specific type of Nash equilibrium problem: fixed linear constraints with structured objective functions.  The criterion proposed in \citet{kojima1985strongly} relied on the degree of Kojima's mapping, which was challenging to determine, thus limiting its applicability to general NEPs. Starting from variational analysis, \citet{rockafellar2018variational} provided a concise equivalent characterization of the strong regularity for $E(p)$ (see Definition \ref{def-strong} in Section \ref{Pre}) by graphical derivatives, where $E(p)$ denotes the Nash equilibria at the perturbation $p$. The KKT solution mapping is not considered in \citet{rockafellar2018variational}. Under convex assumptions, \citet{palomar2010convex} 
were the first to apply the degree theory to the Nikaido-Isoda (NI) function (see \citet{young2014handbook}) of an NEP, effectively obtaining the condition for the robustness of $E(p)$ by leveraging the NEP structure. However, their perturbation model omitted constraint perturbations for each player, and they did not establish the conditions for the robust isolated calmness of $S_{\rm KKT}$ (see Definition \ref{def_calm, isolated calmness} in Section \ref{Pre}).

We focus on three key properties in perturbation theory for Problem (\ref{NEP_perturbations}) at a KKT triple $(\bar{x},\bar{\lambda})$: the strong regularity of $S_{\rm KKT}$ at $\bar{p}$ for $(\bar{x},\bar{\lambda})$ (see Definition \ref{def-strong} in Section \ref{Pre}) , the continuously differentiable single-valued localization of $S_{\rm KKT}$ around $\bar{p}$ for $(\bar{x},\bar{\lambda})$ (see Definition \ref{def-localization} in Section \ref{Pre}), and the robust isolated calmness of $S_{\rm KKT}$ at $\bar{p}$ for $(\bar{x},\bar{\lambda})$ (see Definition \ref{def_calm, isolated calmness} in Section \ref{Pre}), as well as the stability of $\T$. The main contributions of this paper are summarized as follows:

\begin{enumerate}[i)]
    \item Without convex assumptions, we provide, for the first time, equivalent conditions for the strong regularity and the continuously differentiable single-valued localization of $S_{\rm KKT}$ by graphical derivatives, which can be utilized to assess the stability of practical problems. Additionally, to simplify the stability characterizations further, we present two sufficient conditions based on the equivalent characterizations. These sufficient conditions, in the case of a single player, align with the classical stability results in NLP (see \citet{dontchev2020characterizations,dontchev1996characterizations,robinson1975stability}). 
    \item We propose the I-property for the isolated calmness of $S_{\rm KKT}$ by variational analysis without convex assumptions, and the P-property for the robust isolated calmness of $S_{\rm KKT}$ under convex assumptions. The I-property together with the second-order necessary condition (SONC) reduces to the second-order sufficient condition (SOSC) when there is only one player. The P-property, which is stricter than the I-property under convex assumptions, has no existing analogs in the literature to our knowledge. The P-property is also sufficient for the robustness of $E(p)$ with a certain constraint qualification. It reduces to the SOSC when there is only one player. The characterization of robust isolated calmness is more challenging to derive compared to strong regularity and continuously differentiable single-valued localization, as it ensures the existence of the NEs near a known NE without compactness, despite it being the weakest stability requirement for $S_{\rm KKT}$.  
    \item We remark that all stability properties mentioned in Section \ref{Pre} for $\T$ can be deduced for $S_{\rm KKT}$. It is crucial to emphasize that all perturbation models with equality and inequality constraints, to our knowledge, are concerned with full perturbations, where the perturbation parameter $w$ is not fixed (\citet{dontchev2020characterizations,dontchev1996characterizations,robinson1980strongly}). However, in the context of semi-definite programming (SDP) and second-order cone programming (SOCP), typical perturbation models involve only tilt perturbations (\citet{bonnans2005perturbation,ding2017characterization}). Consequently, a natural question arises concerning NEPs. If $\T$ is strongly regular, or possesses robust isolated calmness, and so on, does the same hold for $S_{\rm KKT}$? The answer is affirmative. We provide detailed discussions in the subsequent sections.
\end{enumerate}
The proofs mainly employ two mathematical tools: variational analysis and the degree theory. We refer to \citet{cho2006topological,facchinei2003finite,mordukhovich2015graphical,palomar2010convex,rockafellar2009variational} for more details. For better understanding, we note that the locally continuously single-valued localization of $S_{\rm KKT}$ implies the strong regularity of $S_{\rm KKT}$. In turn, the strong regularity of $S_{\rm KKT}$ implies the robust isolated calmness of $S_{\rm KKT}$.

While research on the stability of NEPs remains scarce, studying the stability of optimization problems holds significant practical value. It plays a pivotal role in the analysis of optimization algorithms (\citet{dontchev2009implicit}). The literature on perturbation analysis of general optimization problems is enormous, and even a short summary about the most important results is beyond the scope of this work. For recent works covering many topics in perturbation analysis, one may refer to \citet{dontchev2009implicit,facchinei2003finite,klatte2002nonsmooth,rockafellar2009variational} and reference therein. For classical models in mathematical programming, such as NLP, SDP, and SOCP, we refer to \citet{bonnans2005perturbation,ding2017characterization,dontchev2020characterizations,fiacco1990nonlinear,robinson1980strongly}.

The remaining parts of this paper are organized as follows. In Section \ref{Pre}, we introduce key definitions and preliminary results on variational analysis and the degree theory. In Section \ref{Strong}, we study the strong regularity of $S_{\rm KKT}$. In Section \ref{Diff}, the characterizations of the continuously differentiable single-valued localization of $S_{\rm KKT}$ are established. The criterion of the robust isolated calmness of $S_{\rm KKT}$ is provided in Section \ref{Robust}. Section \ref{QP} proposes detailed characterizations of the stability of an NEP in which each individual player solves a quadratic programming (QP) problem. We conclude our paper in Section \ref{Con}.

\section{Preliminaries}\label{Pre}
 
Since $S_{\rm KKT}$ is a set-valued mapping, to describe its continuity properties, let us recall some common notations and definitions. 
In this paper, unless causing ambiguity, we use $0$ to represent zero matrices and zero vectors. Let $\Phi :  \mathcal{X} \rightrightarrows \mathcal{Y}$ be a set-valued mapping with $ (\bar{a},\bar{b})\in {\rm gph}\, \Phi$, i.e., $\bar{b}\in \Phi (\bar{a})$, where ${\rm gph}\, \Phi$ denotes the graph of $\Phi$ (see \citet{rockafellar2009variational}). The vector spaces $\mathcal{X}$ and $\mathcal{Y}$ are finite-dimensional Euclidean spaces equipped with the norms $\Vert \cdot \Vert_\mathcal{X}$ and $\Vert \cdot \Vert_\mathcal{Y}$, respectively. 
Based on the equivalence of norms in finite-dimensional Euclidean spaces, without loss of generality, we do not explicitly distinguish between different norms in different spaces in this paper. Instead, we uniformly use $\Vert \cdot \Vert$ to represent them, unless specifically stated.  Let $\mathbb{B}$ be the unit ball in $\mathcal{Y}$. Then the set-valued mapping $\Phi$ is said to be lower semi-continuous (in Berge's sense \citep{berge1963espace}) at $\bar{a}$ for $\bar{b}$ if for any open neighborhood $V$ of $\bar{b}$, there exists an open neighborhood $U$ of $\bar{a}$ such that for all $a\in U$,
$
        \Phi  (a) \cap V \neq \emptyset.
$

There are three Lipschitz-like properties we are interested in for a set-valued mapping, including the Aubin property, isolated calmness and robust isolated calmness.

\begin{definition}\label{def_Aubin}
    A set-valued mapping $\Phi :  \mathcal{X} \rightrightarrows \mathcal{Y}$ is said to have the Aubin property at $\bar{a}\in \mathcal{X}$ for $\bar{b}\in \mathcal{Y}$ if $\bar{b}\in \Phi (\bar{a})$, the graph of $\Phi$ is locally closed at $ (\bar{a},\bar{b})$, and there is a constant $\kappa\ge 0$ along with open neighborhoods $U$ of $\bar{a}$ and $V$ of $\bar{b}$ such that 
    $$
        \Phi (a^\prime)\cap V \subset \Phi (a)+ \kappa \Vert a^\prime -a \Vert \mathbb{B}\quad  {\rm for}\ {\rm all}\ a,a^\prime \in U. 
    $$
\end{definition}

\begin{definition}\label{def_calm, isolated calmness}
    A set-valued mapping $\Phi :  \mathcal{X} \rightrightarrows \mathcal{Y}$ is said to be calm at $\bar{a}$ for $\bar{b}$ if $ (\bar{a},\bar{b})\in {\rm gph}\, \Phi$, and there is a constant $\kappa\ge 0$ along with open neighborhoods $U$ of $\bar{a}$ and $V$ of $\bar{b}$ such that 
    $$
    \Phi  (a) \cap V \subset \Phi (\bar{a}) + \kappa \Vert a-\bar{a} \Vert \mathbb{B} \quad {\rm for}\ {\rm all}\ a\in U.
    $$
    The set-valued mapping $\Phi$ is said to be isolated calm at $\bar{a}$ for $\bar{b}$ if $ (\bar{a},\bar{b})\in {\rm gph}\, \Phi$, and there is a constant $\kappa\ge 0$ along with open neighborhoods $U$ of $\bar{a}$ and $V$ of $\bar{b}$ such that
    $$
    \Phi  (a) \cap V \subset \left\{\bar{b}\right\} + \kappa \Vert a-\bar{a} \Vert \mathbb{B} \quad {\rm for}\ {\rm all}\ a\in U.
    $$
     Furthermore, if for each $a\in U$, $\Phi  (a) \cap V \neq \emptyset$, $\Phi$ is said to be robust isolated calm at $\bar{a}$ for $\bar{b}$.
\end{definition}

The Aubin property defined in Definition \ref{def_Aubin} was designated ``pseudo-Lipschitzian'' by \citet{aubin1984lipschitz}. The calmness for the set-valued mapping $\Phi$ comes from \citet{dontchev2009implicit} and it was called ``upper Lipschitzian'' by \citet{robinson2009generalized}. The isolated calmness (\citet{ding2017characterization,dontchev2009implicit}) for the set-valued mapping $\Phi$ was referred to differently in the literature, e.g., the local upper-Lipschitz continuity in \citet{dontchev2020characterizations}. Moreover, the robust isolated calmness from \citet{ding2017characterization}, was also referred to \citet{dontchev2020characterizations} as locally nonempty-valued property with local upper-Lipschitz continuity.

\begin{remark}\label{iso no robust}
    It is worth noting that $\Phi$ being isolated calm at $\bar{a}$ for $\bar{b}$ does not imply the robust isolated calmness of $\Phi$ at $\bar{a}$ for $\bar{b}$ (see \citet{mordukhovich2015graphical}).
\end{remark}

For a set-valued mapping $\Phi$, there is also a very special property, namely, its graph is the same as that of a single-valued mapping locally. Based on this, we provide the following definition.

\begin{definition}[\citet{dontchev2009implicit}]\label{def-localization}
    For a set-valued mapping $\Phi$, we say that $\Phi$ has a single-valued localization around $\bar{a}$ for $\bar{b}$ if $\bar{b}\in \Phi (\bar{a})$, and there exist neighborhoods $U$ of $\bar{a}$, $V$ of $\bar{b}$ such that ${\rm gph}\, \Phi \cap U\times V$ is a graph of a single-valued mapping.
\end{definition}

Similar to Robinson's implicit function theorem in \citet{robinson1980strongly}, Dontchev and Rockafellar extend it to a generalized equation \citep[Theorem 3F.9]{dontchev2009implicit}. This, combined with the local monotone property of $L_{\rm KKT}$  (\citet[Theorem 3G.5]{dontchev2009implicit}), implies the following lemma.

\begin{lemma}\label{lemma-Aubin}
    The following are equivalent:
    \begin{enumerate}[i)]
        \item $S_{\rm KKT}$ has the Aubin property at $\bar{p}$ for $ (\bar{x},\bar{\lambda})$;
        \item $L_{\rm KKT}$ has the Aubin property at $(\bu,\bv)$ for $ (\bar{x},\bar{\lambda})$;
        \item  $L_{\rm KKT}$ has a Lipschitz continuous single-valued localization around $(\bu,\bv)$ for $ (\bar{x},\bar{\lambda})$;
        \item $S_{\rm KKT}$ has a Lipschitz continuous single-valued localization around $\bar{p}$ for $ (\bar{x},\bar{\lambda})$.
    \end{enumerate}
\end{lemma}

\begin{definition}\label{def-strong}
    For a set-valued mapping $\Phi$, we say that $\Phi$ is strongly regular at $\bar{a}$ for $\bar{b}$ if $\bar{b}\in \Phi (\bar{a})$, and $\Phi$ has a Lipschitz continuous single-valued localization around $\bar{a}$ for $\bar{b}$.
\end{definition}

Based on Lemma \ref{lemma-Aubin} and Definition \ref{def-strong}, it is clear that $S_{\rm KKT}$ has a Lipschitz continuous single-valued localization around $\bar{p}$ for $ (\bar{x},\bar{\lambda})$ if and only if $L_{\rm KKT}$ shares a similar property. The isolated calmness of $S_{\rm KKT}$ and $L_{\rm KKT}$ also exhibits a similar equivalence  (\citet[Corollary 2.3]{dontchev2020characterizations}).

\begin{lemma}\label{lem-iso}
    The following are equivalent:
    \begin{enumerate}[i)]
        \item $S_{\rm KKT}$ is isolated calm at $\bar{p}$ for $ (\bar{x},\bar{\lambda})$;
        \item There exists an open neighborhood $V\times W$ such that $L_{\rm KKT}(\bu,\bv)\cap V\times W = \left\{ (\bar{x},\bar{\lambda})\right\}$.
    \end{enumerate}
\end{lemma}

Different from $S_{\rm KKT}$, $L_{\rm KKT}$ can be reformulated as an inverse of a set-valued mapping. Consequently, we are mainly focusing on the stability of $L_{\rm KKT}$ instead of $S_{\rm KKT}$. 
 For a set-valued mapping $\Psi : \mathcal{Y} \rightrightarrows \mathcal{X}$ and a pair $ (\bar{b},\bar{a}) \in {\rm gph}\,   \Psi$ at which ${\rm gph}\,  \Psi$  is locally closed, the graphical coderivative of $\Psi$ at $\bar{b}$ for $\bar{a}$ is the mapping $D^* \Psi (\bar{b}\, | \,  \bar{a}) : \mathcal{X} \rightrightarrows \mathcal{Y}$ defined by
$$
        \mu\in D^* \Psi (\bar{b}\, | \,\bar{a}) (\nu) \quad \Longleftrightarrow \quad  (\mu,-\nu)\in N_{{\rm gph}\, \Psi} (\bar{b},\bar{a}).
$$

\begin{lemma}[\citet{rockafellar2009variational}]\label{lem-gra}
    For a mapping $\Psi : \mathcal{Y} \rightrightarrows \mathcal{X}$ and a pair $ (\bar{b},\bar{a}) \in {\rm gph}\,  \Psi$ at which ${\rm gph}\, \Psi$  is locally closed, $\Psi^{-1}$ has the Aubin property at $\bar{a}$ for $\bar{b}$ if and only if 
    $$
        0\in D^* \Psi ( \bar{b}\, | \, \bar{a}) (\nu) \Longrightarrow \nu=0.
    $$

\end{lemma}

\begin{lemma}[\citet{rockafellar2009variational}]\label{lem-f+g gra}
    For a function $\psi: \mathcal{Y} \rightarrow \mathcal{X}$ which is continuously differentiable around $\bar{b}$ and a set-valued mapping $\Psi: \mathcal{Y} \rightrightarrows \mathcal{X}$ with $\bar{b}\in \Psi (\bar{a})$, 
    $$
    D^*  \left(\psi+\Psi\right) \left(\bar{b} \, | \, \psi (\bar{b})+\bar{a}\right) (\nu) = \mathcal{J}_b \psi (\bar{b})^{\mathrm{T}}\nu + D^* \Psi  (\bar{b}\, | \, \bar{a}) (\nu)\quad {\rm for}\ {\rm all}\ \nu\in \mathcal{X}.
    $$
\end{lemma}

For an NEP, it is clear that $N=1$ implies that there is only one player. Then $S_{\rm KKT}$ is simplified to the KKT solution mapping of a standard nonlinear programming problem. Thus, we present key results on NLP that are needed for our subsequent discussions. For more details on the second-order optimality conditions and the stability of NLP, see \citet{nocedal1999numerical} and \citet{dontchev2020characterizations}.

Under the case of $N=1$, we suppose that $n =n_1$, $s = s_1$, $d=d_1$, $m=m_1$, $p=p^1$, $x = x^1$, $f = f^1$, $g_i = g^1_i$ and $L = L^1$ for $i=1,...,m$ without loss of generality. Then the NEP is the following standard nonlinear programming problem with canonical perturbations
\begin{equation}\label{NLP_perturbations}
    \begin{array}{ll}
     \min\limits_{x\in \mathbb{R}^n} & f  (x;w)-\left< v,x\right>\\
       {\rm s.\ t.}& g_i (x;w) = u_i, \  i=1,...,s,\\
       & g_j  (x;w) \le u_j,\ j=s+1,...,m.\\
   \end{array}
\end{equation}
The ordinary Lagrangian function is 
$$
    L (x, \lambda ;w ) = f (x; w) + \sum_{i=1}^{s} \lambda_i g_i  (x;w) + \sum_{j=s+1}^{m}\lambda_j g_j  (x;w).
$$
Let 
$$
		G (x;w) = 
		\begin{pmatrix}
			g_1 (x;w)\\
			\vdots\\
			g_m (x;w)
		\end{pmatrix}.
$$
The solution set of 
\begin{equation}\label{NLP-KKT}
\left\{\begin{array}{rl}
    v & = \nabla_x L(x,\lambda;w),\\
    -u &\in -G(x;w) + N_{\mathbb{R}^{s} \times \mathbb{R}_+^{m-s}}(\lambda)
\end{array}\right.
\end{equation}
at $p$ is denoted by  $S^1_{\rm KKT} (p)$. 
We associate with the fixed point $(\bp,\bx,\blam)\in \gph\, \S^1$ the index sets $I_1,I_2,I_3$ in $\{1,2,...,m\}$ defined by
\begin{equation}\label{def-index}
\begin{array}{ll}
& I_1 : = \left\{i\in[s+1,m] \, | \, \bar{\lambda}_i > 0  = g_i  (\bar{x}, \bar{w})-\bar{u}_i  \right\}\bigcup \left\{ 1,...,s\right\},\\
& I_2 : = \left\{i\in[s+1,m] \, | \, \bar{\lambda}_i = 0  = g_i  (\bar{x}, \bar{w})-\bar{u}_i  \right\},\\
& I_3 : = \left\{i\in[s+1,m] \, | \, \bar{\lambda}_i = 0  > g_i  (\bar{x}, \bar{w})-\bar{u}_i  \right\}.
\end{array}
\end{equation}
Recall that the strict Mangasarian-Fromovitz constraint qualification (SMFCQ) holds at $ (\bar{p},\bar{x},\bar{\lambda})$ (see \citet{kyparisis1985uniqueness}) if 
\begin{enumerate}[i)]
    \item The vectors $\nabla_{x} g_i  (\bar{x};\bar{w})$ for $i\in I_1$ are linearly independent;
    \item There is a vector $y\in \mathbb{R}^n$ such that 
    $$
        \left\{\begin{array}{ll}
             \nabla_{x} g_i  (\bar{x};\bar{w})^{\mathrm{T}} y =0 \quad {\rm for}\ {\rm all}\ i\in I_1,\\
             \nabla_{x} g_i  (\bar{x};\bar{w})^{\mathrm{T}} y <0 \quad {\rm for}\ {\rm all}\ i\in I_2.
        \end{array}\right.
    $$
\end{enumerate}
It is known that the SMFCQ holds at $ (\bar{p},\bar{x},\bar{\lambda})$ if and only if there is a unique multiplier vector $\bar{\lambda}$ associated with $ (\bar{p},\bar{x})$  (see \citet{kyparisis1985uniqueness}). Define the critical cone at $(\bar{p},\bar{x},\bar{\lambda})$ as
$$
    \mathcal{C} (\bar{p},\bar{x},\bar{\lambda}): = 
    \left\{y\in \mathbb{R}^{n} \, \bigg{|}  \,
    \begin{array}{ll}    
    & \nabla_{x} g_i  (\bar{x};\bar{w})^{\mathrm{T}} y =0\quad {\rm for}\ {\rm all}\ i\in I_1 \\ 
    &\nabla_{x} g_i  (\bar{x};\bar{w})^{\mathrm{T}} y \le 0\quad {\rm for}\ {\rm all}\  i\in I_2
\end{array}\right\}.
$$
\begin{theorem}[\citet{nocedal1999numerical}]\label{thm-SOSC}
    Suppose that for some feasible point $\bar{x}\in \mathbb{R}^n$, there is a Lagrange multiplier vector $\bar{\lambda}$ such that the KKT conditions  (\ref{NLP-KKT}) are satisfied for $\bar{p}$. Suppose also that 
    \begin{equation}\label{SOSC}
        y^{\mathrm{T}} \nabla^2_{x x}L (\bar{x},\bar{\lambda};\bar{w})y >0 \quad {\rm for}\ {\rm all}\ 0\neq y\in \mathcal{C} (\bar{p},\bar{x},\bar{\lambda}).
    \end{equation}
    Then $\bar{x}$ is a strict local minimizer of (\ref{NLP_perturbations}).
\end{theorem}

The condition in Theorem \ref{thm-SOSC} is referred to as the second-order sufficient condition for optimality. The strong second-order sufficient condition (SSOSC) for optimality holds at $(\bar{p},\bar{x},\bar{\lambda})$ if $y\in \mathcal{C} (\bar{p},\bar{x},\bar{\lambda})$ in  (\ref{SOSC}) is replaced by 
$$
y\in \left\{ y\in \mathbb{R}^n \mid \nabla_{x} g_i  (\bar{x};\bar{w})^{\mathrm{T}} y =0 \quad {\rm for}\ {\rm all}\ i\in I_1 \right\}.
$$

Let
$$
\mathcal{K} := \underbrace{\left\{0\right\} \times \cdots \times \left\{0\right\}}_{n} \times\, \mathbb{R}^{s_1}\times \mathbb{R}_+^{m_1-s_1}\times \cdots \times \mathbb{R}^{s_N}\times \mathbb{R}_+^{m_N-s_N}.
$$ 
To employ Lemma \ref{lem-gra} and \ref{lem-f+g gra} in establishing the strong regularity of $S_{\rm KKT}$, we need the exact characterization of $D^*N_{{\mathcal{K}}}\left((\bar{x},\bar{\lambda})\, | \, -\textbf{L}(\bar{x},\bar{\lambda};\bar{w})\right)$ (see the proof of Theorem \ref{thm-strong-eq}), which can be simplified to the computations of $N_{{\rm gph}\, N_{\mathbb{R}^{r_1} \times \mathbb{R}^{r_2}_+}}(a,b)$ for any $r_1,r_2\in \mathbb{N}_+$ with $(a,b)\in {\rm gph}\,  N_{\mathbb{R}^{r_1} \times \mathbb{R}^{r_2}_+} $.

\begin{proposition}[\citet{mordukhovich2007coderivative}]\label{prop-graOfN}
    Let $(a,b)\in {\rm gph}\, N_{\mathbb{R}^{r_1}\times \mathbb{R}^{r_2}_+}$. Then $(\mu,\nu)\in N_{{\rm gph}\, N_{\mathbb{R}^{r_1}\times \mathbb{R}^{r_2}_+}}(a,b)$ if and only if 
    $$
        \mu_i \in 
        \left\{\begin{array}{lll}
             &  \left\{0\right\} & {\rm when}\ i\in \left\{i\in [r_1+1,r_1+r_2]\, |\, a_i+b_i>0\right\} \bigcup \left\{ 1,...,r_1\right\},\\
             &  \left\{0\right\} \  {\rm or}\ \mathbb{R}_- \  {\rm or}\ \mathbb{R}& {\rm when}\ i\in \left\{i\in [r_1+1,r_1+r_2]\, |\, a_i+b_i=0\right\},\\
             & \mathbb{R} & {\rm when}\ i\in \left\{i\in [r_1+1,r_1+r_2]\, |\, a_i+b_i<0\right\},
        \end{array}\right.
    $$
    and 
    $$
        \nu_i \in 
        \left\{\begin{array}{lll}
             &  \mathbb{R} & {\rm when}\  i\in \left\{i\in [r_1+1,r_1+r_2]\, |\, a_i+b_i>0, \ {\rm or}\ a_i+b_i = 0\ {\rm with}\  \mu_i \in \left\{0\right\}\right\} \bigcup \left\{ 1,...,r_1\right\},\\
             &  \mathbb{R}_+  & {\rm when}\ i\in \left\{i\in [r_1+1,r_1+r_2]\, |\,  a_i+b_i = 0 \ {\rm with}\ \mu_i \in \mathbb{R}_-\right\}, \\
             & \left\{0\right\} & {\rm when}\  i\in \left\{i\in [r_1+1,r_1+r_2]\, |\, a_i+b_i<0, \ {\rm or}\ a_i+b_i = 0\ {\rm with}\ \mu_i\in \mathbb{R}\right\}.
        \end{array}\right.
    $$
\end{proposition}

In the subsequent discussions, we aim to establish the conditions under which NEs exist around $\bar{p}$. In this context, the degree theory becomes indispensable (see \citet{facchinei2003finite,gibbons1992game,myerson2013game,nash2024non}). For comprehensive understanding of the degree theory, we refer to \citet{cho2006topological}. Let $V$ be an open bounded subset in $\mathbb{R}^n$, where ${\rm cl\,} V$ and ${\rm bd\,}V$ denote the closure and the boundary of $V$, respectively.
\begin{lemma}[\citet{cho2006topological}]\label{lem-deg}
    Let $V$ be an open bounded subset of $\mathbb{R}^n$, and let $H: {\rm cl\,} V \rightarrow \mathbb{R}^n$ be a continuous mapping. If $0\notin {\rm bd\,}V$, then there exists an integer ${\rm deg}\,(H,V,0)$ satisfying the following properties:
    \begin{enumerate}[i)]
        \item ${\rm deg}\, (I, V, 0)=1$ if and only if $0\in V$, where $I$ denotes the identity mapping;
        \item If ${\rm deg}\,(H,V,0)\neq 0$, then the equation $H(x) = 0$ has a solution in $V$;
        \item If $H_t(x): [0,1]\times {\rm cl\,} V \rightarrow \mathbb{R}^n$ is continuous and $0\notin \bigcup\limits_{t\in [0,1]} H_t({\rm bd\,}V)$, then ${\rm deg}\,(H_t(\cdot),V,0)$ does not depend on $t\in [0,1]$;
        \item If $\hat{H}: {\rm cl\,} V \rightarrow \mathbb{R}^n$ is continuous, then ${\rm deg}\,(H,V,0) = {\rm deg}\,(\hat{H},V,0)$ if 
        $$
            \max\limits_{x\in {\rm cl\,} V}\Vert H(x)-\hat{H}(x) \Vert_\infty < \inf\limits_{z\in H({\rm bd\,}V)} \Vert z \Vert_\infty,
        $$
        where $\Vert \cdot \Vert_{\infty}$ denotes the infinity norm in $\mathbb{R}^n$.
    \end{enumerate}
\end{lemma}

\section{The strong regularity of $S_{\rm KKT}$}\label{Strong}

By \citet{dontchev2020characterizations}, 
for a nonlinear programming problem with canonical perturbations (\ref{NLP_perturbations}), 
the linear independence constraint qualification (LICQ) and the SSOSC for local optimality 
hold at $(\bar{p},\bar{x},\bar{\lambda})$ if and only if $S^1_{\rm KKT}$ has a 
Lipschitz continuous single-valued localization around $\bar{p}$ for $(\bar{x},\bar{\lambda})$ 
with $x(p)\in X_{\rm KKT}(p)$ being a local minimizer. Since $S_{\rm KKT}$ is composed of 
$N$ instances of $S^1_{\rm KKT}$, a natural question arises that does $S_{\rm KKT}$ 
have a Lipschitz continuous single-valued localization around $\bar{p}$ for $(\bar{x},\bar{\lambda})$ 
if the LICQ and the SSOSC hold for each player $k$ at $(\bar{p}^k,\bar{x}^k,\bar{\lambda}^k)$? 
The answer is negative.

	\begin{example}\label{NEP-exm-1}
        Consider the following NEP with canonical perturbations:
		\begin{equation}\label{perturbation_ex}
			\begin{array}{lllll}
			&{\rm Player}\ 1\quad
			\left\{\begin{array}{ll}
			  \min\limits_{x\in \mathbb{R}} & x^2+3xy-\varepsilon x\\
              {\rm s.\ t.} & x\ge 0,
			\end{array}\right.\qquad \qquad
            
			&{\rm Player}\ 2\quad
			\left\{\begin{array}{ll}
			  \min\limits_{y\in \mathbb{R}} & \frac{3}{2}y^2+2xy+2\varepsilon y\\
              {\rm s.\ t.} & y\le 0.
			\end{array}\right.
			\end{array}
			\end{equation}
	\end{example}
When $\varepsilon=0$, the optimization problem (\ref{perturbation_ex}) has a stationary point $(\bar{x},\bar{y})= (0,0)$. For each player, finding the best strategy is a strictly convex optimization problem. It is clear that the LICQ and the SSOSC hold at $(\bar{\varepsilon},\bar{x},\bar{\lambda})= (0,0,0)$ and $(\bar{\varepsilon},\bar{y},\bar{\mu})=(0,0,0)$. However, for a given $\varepsilon>0$, $(x(\varepsilon),y(\varepsilon))$ is a stationary point if and only if it is a solution to the following generalized equation system
\begin{equation}\label{ex1_KKT}
    \left\{\begin{array}{ll}
        2x(\varepsilon)+3y(\varepsilon)- \varepsilon-\lambda=0, \\
        3y(\varepsilon)+2x(\varepsilon)+ 2\varepsilon+\mu=0,\\
        0\le \lambda \perp x(\varepsilon) \ge 0,\\
        0\le \mu \perp -y(\varepsilon) \ge 0.
    \end{array}\right.
\end{equation}
The generalized equation system (\ref{ex1_KKT}) has no solutions, which implies that $S_{\rm KKT}$ does not have a Lipschitz continuous single-valued localization around $\bar{\varepsilon} = 0$ for $(\bar{x},\bar{y},\bar{\lambda},\bar{\mu})=(0,0,0,0)$. 

In fact, the SSOSC at $(\bar{p}^k,\bar{x}^k,\bar{\lambda}^k)$ is even not a necessary condition. 

\begin{example}\label{NEP-exm-2}
Consider the following NEP with canonical perturbations:
    \begin{equation}\label{perturbation_ex2}
        \begin{array}{lll}
            {\rm Player}\ 1\quad\ 
           & \min\limits_{x,y} & \frac{1}{2}(x+y)^2-xz-2yz-\eps x-2\eps y,
           ~\\
           {\rm Player}\ 2\quad\ 
           & \min\limits_{z} & \frac{1}{2}z^2-xz.
           \end{array}
        \end{equation}
\end{example}
When $\varepsilon = 0$, it is clear that $(\bx,\bar{y},\bar{z}) = (0,0,0) \in \S (\bar{\eps}) = \S (0)$. Based on the subsequent discussions (see Remark \ref{rmk-strong-eq}), $S_{\rm KKT}$ has a Lipschitz continuous single-valued localization around $\bar{\eps}$ for $(\bar{x},\bar{y},\bar{z})$. In fact, $\S(\eps) = (x(\eps),y(\eps),z(\eps)) = (-\eps,\eps,-\eps)$ for any $\eps \in \mathbb{R}$. But the SSOSC at $(\bar{x},\bar{y})= (0,0)$ for Player 1 fails to hold, which requires that the Hessian of Player 1's objective function at $(0,0)$ be positive definite. This comes from the interdependencies introduced by the NEP \eqref{perturbation_ex2} among the players. Specifically, at the point $z(\eps) = - \eps$, Player 1 has multiple optimal strategies, only requiring that $x(\eps)+y(\eps) =0$. This makes the SSOSC fail to hold at $(0,0)$. Nevertheless, Player 2 also needs to find an optimal strategy, which requires $x(\eps) = z(\eps) = -\eps$, thereby ensuring that Player 1 has a unique optimal strategy.

We now provide exact characterizations for the strong regularity of $S_{\rm KKT}$.
 For $k=1,...,N$, we associate with the fixed point $(\bp,\bx,\blam)\in \gph\, \S$ the index sets $I^k_1,I^k_2,I^k_3$ in $\{1,2,...,m_k\}$ defined by
\begin{equation}\label{def-index-NEP}
\begin{array}{ll}
& I^k_1 : = \left\{i\in [s_k+1,m_k] \, | \, \bar{\lambda}^k_i > 0  = g^k_i (\bar{x}^k; \bar{w}^k)-\bar{u}^k_i  \right\}\bigcup \left\{ 1,...,s_k\right\},\\
& I^k_2 : = \left\{i\in [s_k+1,m_k] \, | \, \bar{\lambda}^k_i = 0  = g^k_i (\bar{x}^k; \bar{w}^k)-\bar{u}^k_i  \right\},\\
& I^k_3 : = \left\{i\in [s_k+1,m_k] \, | \, \bar{\lambda}^k_i = 0  > g^k_i (\bar{x}^k; \bar{w}^k)-\bar{u}^k_i  \right\}.
\end{array}
\end{equation}
Since an NLP problem can be considered a special case of an NEP with a single player, we use the same notations $I_1,I_2,I_3$ in \eqref{def-index} to represent the consolidations of $I_1^k,I_2^k,I_3^k$, i.e., $I_1: = \prod_{k=1}^N I^k_1, I_2: = \prod_{k=1}^N I^k_2$, and $I_3: = \prod_{k=1}^N I^k_3$.
Define the index partition set 
$$
    \mathcal{I}: = \left\{ (J_1,J_2,J_3) 
    \, \bigg{|}  \,
            \begin{array}{ll}    
            & J_1 = \prod_{k=1}^N J^k_1, J_2 = \prod_{k=1}^N J^k_2, J_3 = \prod_{k=1}^N J^k_3,\ {\rm where}\ {\rm for}\ k=1,...,N,\\ 
            & J^k_1 \bigcup J^k_2 \bigcup J^k_3= \{1,...,m_k\} \ {\rm with}\
            I^k_1 \subset J^k_1 \subset I^k_1 \bigcup I^k_2, I^k_3 \subset J^k_3 \subset I^k_2\bigcup I^k_3\\
            \end{array}
    \right\}.
$$
Then for any partition $(J_1,J_2,J_3)\in \mathcal{I}$, let 
$$
            K(J_1,J_2) = \left\{
                y=(y^1,...,y^N)\in \mathbb{R}^{n} \, \bigg{|}  \,
            \begin{array}{ll}    
            & \nabla_{x^k} g^k_i (\bar{x}^k;\bar{w}^k)^{\mathrm{T}} y^k =0 \quad {\rm for}\ {\rm all}\ i\in J^k_1,\\ &\nabla_{x^k} g^k_i (\bar{x}^k; \bar{w}^k)^{\mathrm{T}} y^k \le 0 \quad {\rm for}\ {\rm all}\ i\in J^k_2
            \end{array}\right\}.
        $$
In this paper, to avoid complexity in notations, we denote $\nabla^2_{x^k x^i}L^k(\bar{x}^k,\bar{x}^{-k};\bar{w}^k)$ by $\nabla^2_{x^k x^i}L^k$, and $\mathcal{J}_{x^k} G^k(\bar{x}^k;\bar{w}^k)$ by  $\mathcal{J}_{x^k} G^k$ for $i=1,...,N$, $k=1,...,N$. Let 
$$
\begin{array}{ll}
    \textbf{L} (x,\lambda;w) : = 
\end{array}
\left(\begin{array}{cc} 

     &\nabla_{x^1} L^1  (x^1,x^{-1},\lambda^1;w^1) \\
     & \vdots\\
    & \nabla_{x^N} L^N  (x^N,x^{-N},\lambda^N;w^N) \\
    & -G^1  (x^1; w^1)\\
    & \vdots \\
    & -G^N  (x^N; w^N)

\end{array}\right).
$$
Then we obtain the following theorem.
\begin{theorem}\label{thm-strong-eq}
    Let $ (\Bar{x}^1,\Bar{x}^2,...,\Bar{x}^N,\Bar{\lambda}^1,\Bar{\lambda}^2,...,\Bar{\lambda}^N)\in S_{\rm KKT} (\bar{p})$. Then $S_{\rm KKT} (p)$ has a Lipschitz continuous single-valued localization around $\bar{p}$ for $(\bar{x},\bar{\lambda})$ if and only if
    \begin{enumerate}[i)]
        \item The LICQ holds for each player $k$ at $(\bar{p}^k,\bar{x}^k)$, $k=1,...,N$;
        \item For any $K (J_1, J_2)$ associated with the partition $(J_1,J_2,J_3)\in \mathcal{I}$, if $y\in K (J_1, J_2)$, then
        $$
        \begin{pmatrix}
            \nabla^2_{x^1x^1}L^1 & \cdots & (\nabla^2_{x^N x^1}L^N)^{\mathrm{T}}\\
            \vdots & \cdots & \vdots\\
            (\nabla^2_{x^1x^N}L^1)^{\mathrm{T}} & \cdots & \nabla^2_{x^N x^N}L^N
        \end{pmatrix}
        \begin{pmatrix}
            y^1\\
            \vdots \\
            y^N
        \end{pmatrix}
        \in K (J_1, J_2)^\circ \Longrightarrow y = 0.
        $$
    \end{enumerate}
\end{theorem}
\begin{proof}{Proof.}
    It follows from the computations that 
    $$
    \mathcal{J}_{x,\lambda} \textbf{L}(\bx,\bar{\lambda};\bw)^{\mathrm{T}} = 
    \begin{pmatrix}
        \nabla^2_{x^1x^1} L^1 & \cdots & (\nabla^2_{x^N x^1} L^N)^{\mathrm{T}} & -(\mathcal{J}_{x^1} G^1)^{\mathrm{T}} & \cdots & 0\\
        \vdots & \vdots & \vdots & \vdots & \ddots & \vdots\\
        (\nabla^2_{x^1x^N} L^1)^{\mathrm{T}} & \cdots & \nabla^2_{x^N x^N} L^N & 0  & \cdots & -(\mathcal{J}_{x^N} G^N)^{\mathrm{T}}\\
        \mathcal{J}_{x^1} G^1 & \cdots & 0 & 0  & \cdots & 0\\
        \vdots & \ddots & \vdots & \vdots & \ddots & \vdots\\
        0 & \cdots & \mathcal{J}_{x^N}G^N & 0 & \vdots & 0
    \end{pmatrix}.
    $$
    Recall that
$$
\mathcal{K} = \underbrace{\left\{0\right\} \times \cdots \times \left\{0\right\}}_{n} \times \, \mathbb{R}^{s_1} \times \mathbb{R}_+^{m_1-s_1}\times \cdots \times \mathbb{R}^{s_N}\times \mathbb{R}_+^{m_N-s_N}.
$$ 
 By Lemma \ref{lemma-Aubin}, \ref{lem-gra} and \ref{lem-f+g gra}, $S_{\rm KKT}$ has a Lipschitz continuous single-valued localization around $\bar{p}$ for $(\bar{x},\bar{\lambda})$ if and only if for any $q=(q_1,...,q_{2N})\in \mathbb{R}^{n_1}\times \cdots \times \mathbb{R}^{n_N}\times \mathbb{R}^{m_1}\times \cdots \times \mathbb{R}^{m_N} = \mathbb{R}^{n}\times\mathbb{R}^m$,
    \begin{equation}\label{-L=>q=0}
        -\mathcal{J}_{x,\lambda} \textbf{L}(\bx,\bar{\lambda};\bw)^{\mathrm{T}} q \in D^*N_{{\mathcal{K}}}\left((\bar{x},\bar{\lambda})\, | \, -\textbf{L}(\bar{x},\bar{\lambda};\bar{w})\right)(q) \Longrightarrow q=0.
    \end{equation}
    Consequently, (\ref{-L=>q=0}) holds if and only if 
    $$
        \left(-\mathcal{J}_{x,\lambda} \textbf{L}(\bx,\bar{\lambda};\bw)^{\mathrm{T}} q, -q\right)\in N_{{\rm gph}\,N_{\mathcal{K}}}\left((\bar{x},\bar{\lambda}), -\textbf{L}(\bar{x},\bar{\lambda};\bar{w})\right) \Longrightarrow q=0.
    $$
    From \citet[Proposition 6.41]{rockafellar2009variational}, (\ref{-L=>q=0}) is equivalent to
    \begin{equation}\label{Gen=>q=0}
    \left\{\begin{array}{ll}
    \begin{pmatrix}
        \nabla^2_{x^1x^1} L^1 & \cdots & (\nabla^2_{x^N x^1} L^N)^{\mathrm{T}} & -(\mathcal{J}_{x^1} G^1)^{\mathrm{T}} & \cdots & 0\\
        \vdots & \vdots & \vdots & \vdots & \ddots & \vdots\\
        (\nabla^2_{x^1x^N} L^1)^{\mathrm{T}} & \cdots & \nabla^2_{x^N x^N} L^N & 0  & \cdots & -(\mathcal{J}_{x^N} G^N)^{\mathrm{T}}\\
    \end{pmatrix}
    \begin{pmatrix}
        q_1\\
        \vdots\\
        q_N\\
        q_{N+1}\\
        \vdots\\
        q_{2N}
    \end{pmatrix}= 
        \begin{pmatrix}
            0\\
            \vdots\\
            0
        \end{pmatrix},\\
    (-\mathcal{J}_{x^1} G^1q_1, -q_{N+1})\in N_{{\rm gph}\,N_{\mathbb{R}^{s_1}\times \mathbb{R}^{m_1-s_1}_+}}\left(\bar{\lambda}^1,G^1(\bar{x}^1;\bar{w}^1)\right), \\
    \quad\vdots\\
    (-\mathcal{J}_{x^N} G^Nq_N, -q_{2N})\in N_{{\rm gph}\, N_{\mathbb{R}^{s_N}\times \mathbb{R}^{m_N-s_N}_+}}\left(\bar{\lambda}^N,G^N(\bar{x}^N;\bar{w}^N)\right)
\end{array}\right.
\Longrightarrow q = 0.
\end{equation}
    Let $y=(y^1,...,y^N)=-(q_1,...,q_N)\in\mathbb{R}^n$, $y^\prime = (q_{N+1},...,q_{2N}) \in \mathbb{R}^m$. Then from Proposition \ref{prop-graOfN}, (\ref{Gen=>q=0}) actually requires that for any $K (J_1, J_2)$  associated with the partition $(J_1,J_2,J_3)\in \mathcal{I}$, if $y\in K (J_1, J_2)$, the generalized equation system
    $$
    \left\{\begin{array}{ll}
        y^\prime \in 
       - \left[  \begin{pmatrix}
            \mathcal{J}_{x^1} G^1 & \cdots & 0\\
             \vdots & \ddots & \vdots\\
            0  & \cdots & \mathcal{J}_{x^N} G^N
        \end{pmatrix} K (J_1, J_2)\right]^\circ,\\
    \begin{pmatrix}
        \hes_{x^1x^1} L^1 & \cdots & (\hes_{x^N x^1} L^N)^{\mathrm{T}} \\
        \vdots & \vdots & \vdots \\
        (\hes_{x^1x^N} L^1)^{\mathrm{T}} & \cdots & \hes_{x^N x^N} L^N \\
    \end{pmatrix}
    \begin{pmatrix}
        y^1\\
        \vdots\\
        y^N
    \end{pmatrix}
    + 
    \begin{pmatrix}
        (\mathcal{J}_{x^1} G^1)^{\mathrm{T}} & \cdots & 0\\
         \vdots & \ddots & \vdots\\
        0  & \cdots & (\mathcal{J}_{x^N} G^N)^{\mathrm{T}}
    \end{pmatrix}
    \begin{pmatrix}
    y^{\prime1}\\
        \vdots\\
        y^{\prime N}
    \end{pmatrix}= 
    \begin{pmatrix}
        0\\
        \vdots\\
        0
    \end{pmatrix}
\end{array}\right.
    $$
    has only a zero solution. Since
    $$
    y^\prime \in 
    - \left[  \begin{pmatrix}
         \mathcal{J}_{x^1} G^1 & \cdots & 0\\
          \vdots & \ddots & \vdots\\
         0  & \cdots & \mathcal{J}_{x^N} G^N
     \end{pmatrix} K (J_1, J_2)\right]^\circ 
     \Longleftrightarrow 
     \begin{pmatrix}
        (\mathcal{J}_{x^1} G^1)^{\mathrm{T}} & \cdots & 0\\
         \vdots & \ddots & \vdots\\
        0  & \cdots & (\mathcal{J}_{x^N} G^N)^{\mathrm{T}}
    \end{pmatrix}
    \begin{pmatrix}
        y^{\prime1}\\
        \vdots\\
        y^{\prime N}
    \end{pmatrix} \in -K (J_1, J_2)^\circ,
    $$
   we obtain the conclusion. \Halmos
\end{proof}

\begin{remark}\label{rmk-strong-eq} The following results hold:
    \begin{enumerate}[i)]
        \item When $J_1 = I_1$ and $J_2 = I_2$, $K (J_1, J_2)$ is actually the Cartesian product of the critical cone for each player $k$ at $(\bar{p}^k,\bar{x}^k,\bar{\lambda}^k)$;
        \item Because of the multiple choices of $K(J_1, J_2)$ in Theorem \ref{thm-strong-eq}, utilizing the theorem to access the stability property of a practical problem seems challenging;
        \item When $N=1$, i.e., the NEP reduces to a nonlinear programming problem, Theorem \ref{thm-strong-eq} is the ``critical face condition'' obtained by Dontchev and Rockafellar for NLP \citep{dontchev1996characterizations};
        \item The KKT solution mapping $S_{\rm KKT}$ from Example \ref{NEP-exm-2} has a Lipschitz continuous single-valued localization around $\bar{\eps} = 0$ for $(\bar{x},\bar{y},\bar{z})=(0,0,0)$. In fact, for Example \ref{NEP-exm-2}, the LICQ holds for each player at $(\bar{\eps},\bx,\bar{y})$ and $(\bar{\eps},\bar{z})$, as these are unconstrained optimization problems. Moreover, $K(J_1, J_2) = \mathbb{R}^3$ and 
        $$
        \begin{pmatrix}
            \hes_{x x}L^1 &  \hes_{x y}L^1 & (\hes_{z x}L^2)^{\mathrm{T}}\\
            \hes_{y x}L^1 &  \hes_{y y}L^1 & (\hes_{z y}L^2)^{\mathrm{T}}\\
            (\hes_{x z}L^1)^{\mathrm{T}} &  (\hes_{y z}L^1)^{\mathrm{T}} & \hes_{z z}L^2\\
        \end{pmatrix}=
        \begin{pmatrix}
            1 & 1 & -1\\
            1 & 1 & 0\\
            -1 & -2 & 1\\
        \end{pmatrix},
        $$
        which is of full rank.
    \end{enumerate}
\end{remark}

    Given the challenging applicability of the strong regularity characterization for $S_{\rm KKT}$ in Theorem \ref{thm-strong-eq}, we provide a verifiable sufficient characterization below, which also holds clear practical significance. Let $M^k : = \left\{y^k\in \mathbb{R}^{n_k}\, | \, \nabla_{x^k} g^k_i (\bar{x}^k;\bar{w}^k)^{\mathrm{T}} y^k =0, i\in I^k_1\right\}$, and $M: = M^1\times \cdots \times M^N$. Define $B^k\in \mathbb{R}^{n_k \times (n_k-|I^k_1|)}$ as the matrix whose columns consist of the basis of the subspace $M^k$ when the LICQ holds at $(\bar{p}^k,\bar{x}^k)$.
    
    \begin{theorem}\label{thm-strong-suf}
        Let $(\Bar{x}^1,\Bar{x}^2,...,\Bar{x}^N, \Bar{\lambda}^1,\Bar{\lambda}^2,...,\Bar{\lambda}^N)\in S_{\rm KKT}(\bar{p})$. Suppose that the following requirements are fulfilled:
        \begin{enumerate}[i)]
            \item The LICQ holds for each player $k$ at $(\bar{p}^k,\bar{x}^k)$;
            \item The SSOSC holds for each player $k$ at $(\bar{p}^k,\bar{x}^k,\bar{\lambda}^k)$;
            \item There exists a set of parameters 
            $$
                \left\{ \alpha_{i,j}>0 \, \bigg{|} \,  1\le i,j\le N,\, j\neq i,\, {\rm and}\  \sum\limits_{\substack{ j\neq i}} \alpha_{i,j} = 1\ {\rm for}\ i=1,...,N \right\}
            $$
            such that for any $0\neq y^k\in M^k$ and $A_{ki}: = \nabla^2_{x^kx^i}L^k+ (\nabla^2_{x^ix^k}L^i)^{\mathrm{T}}$, 
            $$
                (y^k)^{\mathrm{T}} 
                \left(\nabla^2_{x^kx^k}L^k -\frac{1}{4 \alpha_{i,k}\alpha_{k,i}}\left(A_{ik}^{\mathrm{T}} B^i \left(({B^i})^{\mathrm{T}} \nabla^2_{x^ix^i}L^i B^i\right)^{-1}({B^i})^{\mathrm{T}}A_{ik}\right)\right)
                y^k>0
            $$
            for each player $i$ and player $k$, where $i\neq k$.
        \end{enumerate}
        Then $S_{\rm KKT}(p)$ has a Lipschitz continuous single-valued localization around $\bar{p}$ for $(\bar{x},\bar{\lambda})$. Moreover, for any $\left(p,x(p),\lambda(p)\right)\in {\rm gph}\, S_{\rm KKT}$ around $(\bar{p},\bar{x},\bar{\lambda})$, $x(p)$ is an LNE of Problem (\ref{NEP_perturbations}).
    \end{theorem}

    \begin{proof}{Proof.}
        To begin, we demonstrate that for any $0\neq y\in M$,
    $$
    \begin{pmatrix} 
        (y^1)^{\mathrm{T}}& \cdots & (y^N)^{\mathrm{T}}
    \end{pmatrix}
    \begin{pmatrix}
        \nabla^2_{x^1x^1}L^1 & \cdots & \nabla^2_{x^1 x^N}L^1\\
        \vdots & \cdots & \vdots\\
        \nabla^2_{x^Nx^1}L^N & \cdots & \nabla^2_{x^N x^N}L^N
    \end{pmatrix}
    \begin{pmatrix}
        y^1\\
        \vdots\\
        y^N
    \end{pmatrix}>0,
    $$
    which is equivalent to the following symmetric matrix 
    $$
    \begin{pmatrix}
        ({B^1})^{\mathrm{T}}\nabla^2_{x^1x^1}L^1B^1 & \cdots & \frac{1}{2}({B^1})^{\mathrm{T}}A_{1N}B^N\\
        \vdots & \cdots & \vdots\\
        \frac{1}{2}({B^N})^{\mathrm{T}}A_{N1}B^1 & \cdots & ({B^N})^{\mathrm{T}}\nabla^2_{x^Nx^N}L^NB^N
    \end{pmatrix}
    $$
    being positive definite. For any $z\in \left\{(z^1,...,z^N) \, | \, z^k\in \mathbb{R}^{n_k-|I^k_1|},\ k=1,...,N\right\}$, 
    $$
    \begin{array}{ll}
        &\begin{pmatrix}
            ({z^1})^{\mathrm{T}}& \cdots & ({z^N})^{\mathrm{T}}
        \end{pmatrix}
        \begin{pmatrix}
            ({B^1})^{\mathrm{T}}\nabla^2_{x^1x^1}L^1B^1 & \cdots & \frac{1}{2}({B^1})^{\mathrm{T}}A_{1N}B^N\\
            \vdots & \cdots & \vdots\\
            \frac{1}{2}({B^N})^{\mathrm{T}}A_{N1}B^1 & \cdots & ({B^N})^{\mathrm{T}}\nabla^2_{x^Nx^N}L^NB^N
        \end{pmatrix}
        \begin{pmatrix}
            z^1\\
            \vdots\\
            z^N
        \end{pmatrix}\\
        =& 
        \sum\limits_i \sum\limits_{k\neq i} 
        \begin{pmatrix}
            ({z^k})^{\mathrm{T}} & ({z^i})^{\mathrm{T}}
        \end{pmatrix}
        \begin{pmatrix}
            \frac{\alpha_{k,i}}{2} ({B^k})^{\mathrm{T}}\nabla^2_{x^kx^k}L^kB^k & \frac{1}{4}({B^k})^{\mathrm{T}}A_{ki}B^i \\
            \frac{1}{4}({B^i})^{\mathrm{T}}A_{ik}B^k& \frac{\alpha_{i,k}}{2}({B^i})^{\mathrm{T}}\nabla^2_{x^ix^i}L^iB^i
        \end{pmatrix}
        \begin{pmatrix}
            z^k\\
            z^i
        \end{pmatrix}.
        \\
    \end{array}
    $$
    Following from the Schur complement, the symmetric matrix 
    $$
    \begin{pmatrix}
        \frac{\alpha_{k,i}}{2} ({B^k})^{\mathrm{T}}\nabla^2_{x^kx^k}L^kB^k & \frac{1}{4}({B^k})^{\mathrm{T}}A_{ki}B^i \\
        \frac{1}{4}({B^i})^{\mathrm{T}}A_{ik}B^k& \frac{\alpha_{i,k}}{2}({B^i})^{\mathrm{T}}\nabla^2_{x^ix^i}L^iB^i
    \end{pmatrix}
    $$
    is positive definite if and only if both $ \frac{\alpha_{k,i}}{2} ({B^k})^{\mathrm{T}}\nabla^2_{x^kx^k}L^kB^k$ and 
    $$
    \frac{\alpha_{i,k}}{2}({B^i})^{\mathrm{T}} \hes_{x^ix^i}L^i B^i - \frac{1}{8\alpha_{k,i}}({B^i})^{\mathrm{T}}\left(A_{ki}^{\mathrm{T}} B^k \left(({B^k})^{\mathrm{T}} \nabla^2_{x^kx^k}L^k B^k\right)^{-1}({B^k})^{\mathrm{T}}A_{ki}\right)B^i
    $$
    are positive definite. These conditions are equivalent to the SSOSC holding at $(\bar{p}^k,\bar{x}^k,\bar{\lambda}^k)$ and for any $0 \neq y^k\in M^k$, 
    $$
    (y^k)^{\mathrm{T}} 
    \left(\nabla^2_{x^kx^k}L^k -\frac{1}{4 \alpha_{i,k}\alpha_{k,i}}\left(A_{ik}^{\mathrm{T}} B^i \left(({B^i})^{\mathrm{T}} \nabla^2_{x^ix^i}L^i B^i\right)^{-1}({B^i})^{\mathrm{T}}A_{ik}\right)\right)
    y^k>0
    $$
    for each player $k$ and player $i$, where $i\neq k$.
    Consequently, for any $0\neq y\in M$,
    $$
    \begin{pmatrix} 
        (y^1)^{\mathrm{T}}& \cdots & (y^N)^{\mathrm{T}}
    \end{pmatrix}
    \begin{pmatrix}
        \nabla^2_{x^1x^1}L^1 & \cdots & \nabla^2_{x^1 x^N}L^1\\
        \vdots & \cdots & \vdots\\
        \nabla^2_{x^Nx^1}L^N & \cdots & \nabla^2_{x^N x^N}L^N
    \end{pmatrix}
    \begin{pmatrix}
        y^1\\
        \vdots\\
        y^N
    \end{pmatrix}>0.
    $$
    For any cone $K(J_1, J_2)$ associated with the partition $(J_1,J_2,J_3)\in \mathcal{I}$, if $y\in K (J_1, J_2) \subset M$ and
    $$
    \begin{pmatrix}
        \nabla^2_{x^1x^1}L^1 & \cdots & (\nabla^2_{x^N x^1}L^N)^{\mathrm{T}}\\
        \vdots & \cdots & \vdots\\
        (\hes_{x^1x^N}L^1)^{\mathrm{T}} & \cdots & \hes_{x^N x^N}L^N
    \end{pmatrix}
    \begin{pmatrix}
        y^1\\
        \vdots\\
        y^N
    \end{pmatrix}\in K(J_1, J_2)^\circ,
    $$
    then 
    $$
    \begin{pmatrix} 
        (y^1)^{\mathrm{T}}& \cdots & (y^N)^{\mathrm{T}}
    \end{pmatrix}
    \begin{pmatrix}
        \nabla^2_{x^1x^1}L^1 & \cdots & \nabla^2_{x^1 x^N}L^1\\
        \vdots & \cdots & \vdots\\
        \nabla^2_{x^Nx^1}L^N & \cdots & \nabla^2_{x^N x^N}L^N
    \end{pmatrix}
    \begin{pmatrix}
        y^1\\
        \vdots\\
        y^N
    \end{pmatrix}\le0,
    $$
    which implies that $y = 0$. From Theorem \ref{thm-strong-eq}, $S_{\rm KKT}(p)$ has a Lipschitz continuous single-valued localization around $\bar{p}$ for $(\bar{x},\bar{\lambda})$. For any $(p,x(p),\lambda(p))\in {\rm gph}\, S_{\rm KKT}$ close enough to $(\bar{p},\bar{x},\bar{\lambda})$, we now prove that $x(p)$ is a local Nash equilibrium. Let the index sets for each player $k$ associated with the point $(p,x(p),\lambda (p))$ be defined as follows:
    $$
    \begin{array}{ll}
    & \tilde{I^k_1} : = \left\{i\in [s_k+1,m_k] \, | \, \lambda^k_i(p) > 0  = g^k_i (x^k(p), w^k)-u^k_i  \right\}\bigcup \left\{ 1,...,s_k\right\},\\
    & \tilde{I^k_2} : = \left\{i\in [s_k+1,m_k] \, | \, \lambda^k_i(p) = 0  = g^k_i (x^k(p), w^k)-u^k_i  \right\},\\
    & \tilde{I^k_3} : = \left\{i\in [s_k+1,m_k] \, | \, \lambda^k_i(p) = 0  > g^k_i (x^k(p), w^k)-u^k_i  \right\}.
    \end{array}     
    $$
    Since the SSOSC holds for each player $k$ at $(\bar{p}^k,\bar{x}^k,\bar{\lambda}^k)$, $I^k_1 \subset \tilde{I^k_1}$ and $\tilde{I^k_2}\subset I^k_2$, the SOSC holds at $(p^k, x^k(p),\lambda^k(p))$, which implies that $x = (x^1(p),...,x^N(p))$ is a local Nash equilibrium (see Theorem \ref{thm-SOSC}). \Halmos

\end{proof}

    \begin{remark}\label{rmk-strong-suf}
        When $N=1$, Theorem \ref{thm-strong-suf} aligns with the well-known result regarding the strong regularity of $S^1_{\rm KKT}$ in NLP, i.e., according to Theorem 4.2 by \citet{dontchev2020characterizations}, the LICQ and the SSOSC imply the strong regularity of $S^1_{\rm KKT}$. 
    \end{remark}

    \begin{remark}\label{rmk-strong-suf-algorithm}
        Theorem \ref{thm-strong-suf} demonstrates that for any $y\in M$, the quadratic form 
        \begin{equation}\label{quadratic form}
        \begin{pmatrix}
            \nabla^2_{x^1x^1}L^1 & \cdots & \nabla^2_{x^1 x^N}L^1\\
            \vdots & \cdots & \vdots\\
            \nabla^2_{x^Nx^1}L^N & \cdots & \nabla^2_{x^N x^N}L^N
        \end{pmatrix}
        \end{equation}
        is positive semi-definite if the SSOSC holds at $(\bp^k,\bx^k,\blam^k)$, and the strict inequality in condition iii) is relaxed to $\ge$ for $k=1,...,N$. This result implies the monotonicity of \eqref{quadratic form} on $M$, which is especially useful for designing algorithms to solve NEPs. For example, Sections 5.2 and 5.3 in \citet{facchinei2010generalized} provide discussions on variational inequality type and Nikaido-Isoda-function type methods for NEPs.
    \end{remark}

   If only $w^k$ remains unchanged for $k=1,...,N$, and $\T$ is strongly regular at $(\bar{u},\bar{v})$ for $(\bar{x},\bar{\lambda})$, is $S_{\rm KKT}$ also strongly regular at $\bar{p}$ for $(\bar{x},\bar{\lambda})$? The following corollary provides a positive answer based on Theorem \ref{thm-strong-eq}. 

    \begin{corollary}\label{coro-strong-til}
        Let $ (\Bar{x}^1,\Bar{x}^2,...,\Bar{x}^N, \Bar{\lambda}^1,\Bar{\lambda}^2,...,\Bar{\lambda}^N)\in S_{\rm KKT} (\bar{p})$. The following are equivalent:
        \begin{enumerate}[i)]
            \item $S_{\rm KKT} (p)$ has a Lipschitz continuous single-valued localization around $\bar{p}$ for $(\bar{x},\bar{\lambda})$;
            \item $\T (u,v)$ has a Lipschitz continuous single-valued localization around $(\bar{u},\bar{v})$ for $(\bar{x},\bar{\lambda})$.
        \end{enumerate}
    \end{corollary}
    
    \begin{proof}{Proof.}
        This is an immediate result from Theorem \ref{thm-strong-eq} or Lemma \ref{lemma-Aubin}. \Halmos
    \end{proof}

    \section{The continuously differentiable single-valued localization of $S_{\rm KKT}$}\label{Diff}
    
    As a consequence of Theorem \ref{thm-strong-eq} and Theorem \ref{thm-strong-suf}, we 
    present a characterization of the continuously differentiable 
    single-valued localization of $S_{\rm KKT}$ which is called the 
    Jacobian uniqueness conditions by \citet{fiacco1990nonlinear} in NLP. 
    Without loss of generality, in this section, we 
    suppose that for $k=1,...,N$, the index sets \eqref{def-index-NEP} associated with the point $(\bp,\bx,\blam)\in \gph \,\S$ are  
        $$
\begin{array}{ll}
& I^k_1  = \left\{1,...,s_k \right\} \bigcup \left\{s_k+1,...,|I^k_1|\right\},\\
& I^k_2  = \left\{|I^k_1|+1,...,|I^k_1|+|I^k_2|  \right\},\\
& I^k_3  = \left\{|I^k_1|+|I^k_2|+1,...,  m_k\right\}.
\end{array}
$$
    Let $\Pi_{\mathbb{R}^{|I_1|}}: \mathbb{R}^m \rightarrow \mathbb{R}^{|I_1|}$ be the projection operator such that 
    $$
        \Pi_{\mathbb{R}^{|I_1|}}(\lambda) = (\tlam^1,...,\tlam^N),
    $$ 
    where $\tlam^k = (\lambda^k_1,...,\lambda^k_{|I^k_1|})$ for $k=1,...,N$. Let the projection of $\lambda$ onto $\mathbb{R}^{|I_1|}$ be defined as $\tilde{\lambda} := \Pi_{\mathbb{R}^{|I_1|}} (\lambda)$. Consider the following continuously differentiable mapping 
    \begin{equation}\label{NEP-KKT-Reduce}
        \mathbf{F} (x,\tlam , p): = 
        \begin{pmatrix}
            &\nabla_{x^1} \tilde{L}^1 (x^1,x^{-1},\tilde{\lambda}^1;w^1) - v^1 \\
         & \vdots\\
        & \nabla_{x^N} \tilde{L}^N (x^N,x^{-N},\tilde{\lambda}^N;w^N) -v^N\\
        & -\tilde{G}^1(x^1;w^1) + \tu^1 \\
        & \vdots \\
        & -\tilde{G}^N (x^N;w^N)+ \tu^N\\
        \end{pmatrix},
    \end{equation}
    where for $k=1,...,N$, $\tilde{u}^k = (u^k_1,...,u^k_{|I^k_1|})$,
    $$
    \tilde{G}^k (x^k;w^k) = 
         \begin{pmatrix}
             g^k_1 (x^k;w^k)\\
             \vdots\\
            g^k_{|I^k_1|} (x^k;w^k)\\
        \end{pmatrix},
    $$
    and 
    $$
    \tilde{L}^k(x^k, x^{-k},\tilde{\lambda}^k;w^k) = f^k(x^k,x^{-k};w^k)+ \tilde{G}^k (x^k;w^k)^{\mathrm{T}} \tilde{\lambda}^k.
    $$
    We say that the the strict complementary slackness condition (SCSC) 
    holds for player $k$ at $(\bar{p}^k,\bar{x}^k,\bar{\lambda}^k)$ 
    if $I^k_2 = \emptyset$. Then the characterization of the continuously differentiable single-valued 
    localization of $\S$ is established.

    \begin{theorem}\label{thm-dif-eq}
        Let $ (\Bar{x}^1,\Bar{x}^2,...,\Bar{x}^N, \Bar{\lambda}^1,\Bar{\lambda}^2,...,\Bar{\lambda}^N)\in S_{\rm KKT} (\bar{p})$. The following are equivalent:
        \begin{enumerate}[i)]
            \item $S_{\rm KKT}$ has a continuously differentiable single-valued localization around $\bar{p}$ for $(\bar{x},\bar{\lambda})$;
            \item The SCSC and the LICQ hold for each player $k$ at $(\bar{p}^k,\bar{x}^k,\bar{\lambda}^k)$, $k=1,...,N$. Moreover, for any $y\in M$,
            $$
            \begin{pmatrix}
                \nabla^2_{x^1x^1}L^1 & \cdots & \nabla^2_{x^1 x^N}L^1\\
                \vdots & \cdots & \vdots\\
                \nabla^2_{x^Nx^1}L^N & \cdots & \nabla^2_{x^N x^N}L^N
            \end{pmatrix}
            \begin{pmatrix}
                y^1\\
                \vdots\\
                y^N
            \end{pmatrix}\in M^\perp \Longrightarrow y=0.
            $$
        \end{enumerate}
    \end{theorem}

    \begin{proof}{Proof.}
        $ii)\Longrightarrow i)$: If $ii)$ holds, 
        it can be deduced that the partial derivative of $\mathbf{F}$ with respect to 
    $(x,\tlam)$ at $(\bx,\mypi(\blam),\bp)$   
    \begin{equation}\label{derivative}
        \mathcal{J}_{(x,\tlam)} \mathbf{F} (\bx,\mypi(\blam),\bp)= 
        \begin{pmatrix}
            \nabla^2_{x^1x^1}L^1 & \cdots & \nabla^2_{x^1 x^N}L^1 & (\mathcal{J}_{x^1} \tilde{G})^{\mathrm{T}} & \cdots & 0\\
            \vdots & \cdots & \vdots & \vdots & \ddots & \cdots\\
            \nabla^2_{x^Nx^1}L^N & \cdots & \nabla^2_{x^N x^N}L^N & 0 & \cdots & (\mathcal{J}_{x^N} \tilde{G}^N)^{\mathrm{T}}\\
            -\mathcal{J}_{x^1} \tilde{G}^1 & \cdots & 0 & 0 & \cdots & 0\\
            \vdots & \ddots & \vdots & \vdots & \cdots & \vdots\\
            0 & \cdots & -\mathcal{J}_{x^N} \tilde{G}^N & 0 & \cdots & 0\\
        \end{pmatrix}
        \end{equation}
    is invertible. Then by the classical implicit function theorem, there exist a neighborhood $U_1$ of $\bp$, a neighborhood $V\times \tilde{W}$ of 
    $(\bx,\mypi(\blam))$, and a  
    continuously differentiable function $U_1 \owns p \mapsto (x(p),\tlam(p))\in V\times \tilde{W}$ such that $\mathbf{F}(x(p),\tlam(p), p) = 0$. For $p\in U_1$, $k=1,...,N$, let 
    $$
        \lambda^k_i (p): = \left\{\begin{array}{ll}
            \tilde{\lambda}^k_i (p) & {\rm when}\ i\in\left\{ 1,...,|I^k_1|\right\},\\
            \bar{\lambda}^k_i = 0 & {\rm when}\ i\in \left\{ |I^k_1| +1 ,..., m_k \right\}.
        \end{array}\right.
    $$  
    If $g^k_i (\bar{w},\bar{x})-\bar{u}^k_i<0$, then  $g^k_i (w,x(p))-u^k_i<0$ in a neighborhood of $\bar{p}$ 
    due to the continuity of $g^k_i$. Consequently, there exist  sufficiently small neighborhoods $U\subset U_1$ of $\bp$ and $V\times W$ of $(\bx,\blam)$ such that $(x(p),\lambda(p))\in S_{\rm KKT}(p)\bigcap V\times W$, and the function $U\owns p \mapsto (x(p),\lambda(p))\in V \times W$ is continuously differentiable.
     From Theorem \ref{thm-strong-eq}, $S_{\rm KKT}$ has a Lipschitz continuous single-valued localization around $\bar{p}$ for $(\bx,\blam)$. This combined with the inclusion $(p,x(p),\lambda(p))\in \gph\, \S \cap U\times V \times W$ implies that $S_{\rm KKT}$ has a continuously differentiable single-valued localization around $\bar{p}$.
        
        $i)\Longrightarrow ii)$:  If $S_{\rm KKT}$ has a continuously differentiable single-valued localization around $\bar{p}$ for $(\bar{x},\bar{\lambda})$, it has a Lipschitz continuous single-valued localization around $\bar{p}$ for $(\bar{x},\bar{\lambda})$. Following from Theorem \ref{thm-strong-eq}, it suffices to prove that the SCSC holds for each player $k$ at $(\bar{p}^k,\bar{x}^k,\bar{\lambda}^k)$. Without loss of generality, assume that $i\in I^1_2\neq \emptyset$, and $U \owns p \mapsto (x(p), \lambda (p))\in \S (p) \bigcap V \times W$ is continuously differentiable. Let $p^1(\varepsilon)= (\bar{u}^1_1,...,\bar{u}^1_{i-1}, \bar{u}^1_i + \varepsilon,\bar{u}^1_{i+1},...,\bar{u}^1_{m_1}, \bar{v}^1,\bar{w}^1)$, and $p(\varepsilon)=(p^1(\varepsilon),\bar{p}^2,...,\bar{p}^N)$ for any $\varepsilon>0$ close to $0$. Since $\bar{\lambda}^1_i = 0$, for sufficiently small $\eps>0$, $(x(p(\eps)),\lambda (p(\eps))) = (\bx, \blam)$. Thus, $x^1(p(\varepsilon))$ must satisfy 
        $$
            \frac{d}{d\varepsilon} x^1(p({\varepsilon})) |_{\varepsilon=0} = 0.
        $$
        On the other hand, for any $\varepsilon<0$ close to $0$, the corresponding solution $x^1(p(\varepsilon))$ is feasible, which yields 
        $$
            -\bar{u}^1_i - \varepsilon + g^1_i (x^1(p(\varepsilon));\bar{w}^1) \le 0.
        $$
        As $-\bar{u}^1_i + g^1_i(\bar{x}^1; \bar{w}^1)=0$, we obtain that 
        $$
           \nabla_{x^1} g^1_i (\bar{x}^1 ;\bar{w}^1)^{\mathrm{T}}\frac{d}{d\varepsilon} x^1(p(\varepsilon)) |_{\varepsilon=0}\ge 1,
        $$
       which is a contradiction. \Halmos
    \end{proof}

    By Theorem \ref{thm-strong-suf} and Theorem \ref{thm-dif-eq}, we obtain Corollary \ref{coro-dif-suf}. It imposes stronger requirements on the problem for the continuously differentiable single-valued localization of $\S$, but at the same time, it ensures that the stationary points after perturbations are LNEs.

    \begin{corollary}\label{coro-dif-suf}
        Let $(\Bar{x}^1,\Bar{x}^2,...,\Bar{x}^N, \Bar{\lambda}^1,\Bar{\lambda}^2,...,\Bar{\lambda}^N)\in S_{\rm KKT}(\bar{p})$. Suppose that the following requirements are fulfilled:
        \begin{enumerate}[i)]
            \item The SCSC holds for each player $k$ at $(\bar{p}^k,\bar{x}^k,\bar{\lambda}^k)$;
            \item The LICQ holds for each player $k$ at $(\bar{p}^k,\bar{x}^k)$;
            \item The SOSC holds for each player $k$ at $(\bar{p}^k,\bar{x}^k,\bar{\lambda}^k)$;
            \item There exists a set of parameters 
            $$
                \left\{ \alpha_{i,j}>0 \, \bigg{|} \,  1\le i,j\le N,\, j\neq i,\, {\rm and}\  \sum\limits_{\substack{ j\neq i}} \alpha_{i,j} = 1\ {\rm for}\  i=1,...,N \right\}
            $$
            such that for any $0\neq y^k\in M^k$ and $A_{ki}: = \nabla^2_{x^kx^i}L^k+ (\nabla^2_{x^ix^k}L^i)^{\mathrm{T}}$, 
            $$
                (y^k)^{\mathrm{T}} 
                \left(\nabla^2_{x^kx^k}L^k -\frac{1}{4 \alpha_{i,k}\alpha_{k,i}}\left(A_{ik}^{\mathrm{T}} B^i \left(({B^i})^{\mathrm{T}} \nabla^2_{x^ix^i}L^i B^i\right)^{-1}({B^i})^{\mathrm{T}}A_{ik}\right)\right)
                y^k>0
            $$
            for each player $i$ and player $k$, where $i\neq k$.
        \end{enumerate}
        Then $S_{\rm KKT}$ has a continuously differentiable single-valued localization around $\bar{p}$ for $(\bar{x},\bar{\lambda})$, and for any $(p,x(p),\lambda(p))\in {\rm gph}\, S_{\rm KKT}$ around $(\bar{p},\bar{x},\bar{\lambda})$, $x(p)$ is an LNE of Problem (\ref{NEP_perturbations}).
    \end{corollary}

    \begin{remark}
        When $N=1$, Corollary \ref{coro-dif-suf} corresponds to the well-known result on the continuously differentiable single-valued localization of $S^1_{\rm KKT}$ in NLP, as established by \citet{fiacco1990nonlinear}. This result implies that the LICQ, the SCSC and the SOSC ensure the continuously differentiable single-valued localization of $S^1_{\rm KKT}$.
    \end{remark}

    Through the proof of Theorem \ref{thm-dif-eq}, we observe that, same as strong regularity, if $\T$ has a continuously differentiable single-valued localization around $(\bar{u},\bar{v})$ for $(\bar{x},\bar{\lambda})$, then $S_{\rm KKT}$ also has a continuously differentiable single-valued localization around $\bar{p}$ for $(\bar{x},\bar{\lambda})$. In other words, investigating the strong regularity and the continuously differentiable single-valued localization of $\T$ is equivalent to studying those of $S_{\rm KKT}$.

    \begin{corollary}\label{coro-dif-til}
        Let $ (\Bar{x}^1,\Bar{x}^2,...,\Bar{x}^N, \Bar{\lambda}^1,\Bar{\lambda}^2,...,\Bar{\lambda}^N)\in S_{\rm KKT} (\bar{p})$. The following are equivalent:
        \begin{enumerate}[i)]
            \item $S_{\rm KKT}$ has a continuously differentiable single-valued localization around $\bar{p}$ for $(\bar{x},\bar{\lambda})$;
            \item $\T$ has a continuously differentiable single-valued localization around $(\bar{u},\bar{v})$ for $(\bar{x},\bar{\lambda})$.
        \end{enumerate}
    \end{corollary}

    \begin{proof}{Proof.}
        The implication $i)\Longrightarrow ii) $ is straightforward. It suffices to prove that $ii)\Longrightarrow i)$. Following from Corollary \ref{coro-strong-til},  $S_{\rm KKT} (p)$ has a Lipschitz continuous  single-valued localization around $\bar{p}$ for $(\bar{x},\bar{\lambda})$. Consequently, the LICQ holds for each player $k$ at $(\bar{p}^k,\bar{x}^k)$, $k=1,...,N$. Moreover, for any $y\in M$,
        $$
            \begin{pmatrix}
                \nabla^2_{x^1x^1}L^1 & \cdots & \nabla^2_{x^1 x^N}L^1\\
                \vdots & \cdots & \vdots\\
                \nabla^2_{x^Nx^1}L^N & \cdots & \nabla^2_{x^N x^N}L^N
            \end{pmatrix}
            \begin{pmatrix}
                y^1\\
                \vdots\\
                y^N
            \end{pmatrix}\in M^\perp \Longrightarrow y=0.
            $$
        It suffices to prove that the SCSC holds for each player $k$ at $(\bar{p}^k,\bar{x}^k,\bar{\lambda}^k)$. Since $\T$ has a continuously differentiable single-valued localization around $(\bar{u},\bar{v})$ for $(\bar{x},\bar{\lambda})$, the proof is similar to that of Theorem \ref{thm-dif-eq}. \Halmos
        
        \end{proof}
\section{The robust isolated calmness of $S_{\rm KKT}$}\label{Robust}

Following from Lemma \ref{lem-iso}, the isolated calmness of $S_{\rm KKT}$ at $\bar{p}$ for $(\bar{x},\bar{\lambda})$ can be derived by focusing on the condition when $(x,\lambda) = (\bar{x},\bar{\lambda})$ is an isolated value of $L_{\rm KKT}(\bu,\bv)$. Recall that $L_{\rm KKT}(u,v)$ is the set of all points $(x,\lambda)$ solving the following generalized equation system
\begin{equation}\label{NEP-LKKT-local-calmness}
    \left\{\begin{array}{rl}
        v^1 &= \nabla_{x^1}L^1(\bar{x}^1,\bar{x}^{-1},\bar{\lambda}^1;\bar{w}^1)+ \sum_{i=1}^{N} \nabla^2_{x^1 x^i}L^1(\bar{x}^1,\bar{x}^{-1},\bar{\lambda}^1;\bar{w}^1)(x^i-\bar{x}^{i})\\
        &+\mathcal{J}_{x^1}G^1(\bar{x}^1;\bar{w}^1)^{\mathrm{T}} (\lambda^1-\bar{\lambda}^1),\\
        &\vdots\\
        v^N &= \nabla_{x^N}L^N(\bar{x}^N,\bar{x}^{-N},\bar{\lambda}^N;\bar{w}^N)+ \sum_{i=1}^{N} \nabla^2_{x^N x^i}L^N(\bar{x}^N,\bar{x}^{-N},\bar{\lambda}^N;\bar{w}^N)(x^i-\bar{x}^{i})\\
        &+\mathcal{J}_{x^N}G^N(\bar{x}^N;\bar{w}^N)^{\mathrm{T}} (\lambda^N-\bar{\lambda}^N),\\
        - u^1 &\in -G^1(\bar{x}^1;\bar{w}^1)-\mathcal{J}_{x^1}G^1(\bar{x}^1;\bar{w}^1)(x^1-\bar{x}^1)+N_{\mathbb{R}^{s_1} \times \mathbb{R}_+^{m_1-s_1}}(\lambda^1),\\
        &\vdots\\
        - u^N &\in -G^N(\bar{x}^N;\bar{w}^N)-\mathcal{J}_{x^N}G^N(\bar{x}^N;\bar{w}^N)(x^N-\bar{x}^N)+N_{\mathbb{R}^{s_N} \times \mathbb{R}_+^{m_N-s_N}}(\lambda^N)\\
    \end{array}\right.
\end{equation}
at the given point $(\bp,\bx,\blam)\in \gph\, \S$, and $(\bp,\bx,\blam)$ satisfies the generalized equation system
$$
\left\{\begin{array}{rl}
    \bar{v}^1 &= \nabla_{x^1} L^1  (\bar{x}^1,\bar{x}^{-1},\bar{\lambda}^1;\bar{w}^1),\\
    &\vdots\\
    \bar{v}^N &= \nabla_{x^N} L^N  (\bar{x}^N,\bar{x}^{-N},\bar{\lambda}^N;\bar{w}^N),\\
    -\bar{u}^1 &\in -G^1  (\bar{x}^1; \bar{w}^1) + N_{\mathbb{R}^{s_1} \times \mathbb{R}_+^{m_1-s_1}} (\bar{\lambda}^1),\\
    &\vdots\\
    -\bar{u}^N &\in -G^N(\bar{x}^N;\bar{w}^N)+N_{\mathbb{R}^{s_N} \times \mathbb{R}_+^{m_N-s_N}} (\bar{\lambda}^N).
\end{array}\right.
$$

We introduce two properties below. The first property, along with some basic assumptions, ensures the isolation of $(\bar{x},\bar{\lambda})$. More importantly, the latter property further guarantees the robust isolated calmness of $S_{\rm KKT}$ at $\bar{p}$ for $(\bar{x},\bar{\lambda})$ under a certain constraint qualification and convex assumptions. 

\begin{definition}\label{def-I-pro}
    We say that Problem (\ref{NEP_perturbations}) satisfies {\rm the I-property} at $(\bar{p},\bar{x},\bar{\lambda})$ on $K(I_1,I_2)$ if for any $y\in K(I_1,I_2)$,
    $$
        \sum\limits_{i=1}^N (y^k)^{\mathrm{T}} \nabla^2_{x^kx^i}L^k(\bar{x}^k,\bar{x}^{-k},\bar{\lambda}^k;\bar{w}^k) y^i =0 \ {\rm for}\ k=1,...,N \Longrightarrow y = 0.
    $$
\end{definition}

\begin{definition}\label{def-P-pro}
    We say that Problem (\ref{NEP_perturbations}) satisfies {\rm the P-property} at $(\bar{p},\bar{x},\bar{\lambda})$ on $K(I_1,I_2)$ if for any $0\neq y\in  K(I_1,I_2)$,
    $$
        \max\limits_{k=1,...,N}
        \sum\limits_{i\neq k}^N (y^k)^{\mathrm{T}} \nabla^2_{x^kx^i}L^k(\bar{x}^k,\bar{x}^{-k},\bar{\lambda}^k;\bar{w}^k) y^i + \frac{1}{2}(y^k)^{\mathrm{T}} \nabla^2_{x^k x^k} L^k (\bar{x}^k,\bar{x}^{-k},\bar{\lambda}^k;\bar{w}^k) y^k>0. 
    $$
\end{definition}

\begin{remark}\label{rmk-I-P-pro}
    It is clear that
    \begin{enumerate}[i)]
        \item If $\bar{x}^k$ is a local minimizer for each player $k$ with the associated multiplier $\bar{\lambda}^k$, then {\rm the P-property} $\Longrightarrow$ {\rm the I-property};
        \item In NLP, {\rm the I-property} together with the SONC reduces to the SOSC;
        \item In NLP, {\rm the P-property} reduces to the SOSC.
    \end{enumerate}
    
\end{remark}

We now provide a theorem concerning the isolated calmness of $S_{\rm KKT}$ based on Definition \ref{def-I-pro}.
\begin{theorem}\label{thm-iso}
    Let $(\Bar{x}^1,\Bar{x}^2,...,\Bar{x}^N, \Bar{\lambda}^1,\Bar{\lambda}^2,...,\Bar{\lambda}^N)\in S_{\rm KKT}(\bar{p})$, and the following requirements are fulfilled: 
        \begin{enumerate}[i)]
            \item The SMFCQ holds at $(\bar{p}^k,\bar{x}^k,\bar{\lambda}^k)$ for each player $k$;
            \item Problem (\ref{NEP_perturbations}) satisfies {\rm the I-property} at $(\bar{p},\bar{x},\bar{\lambda})$ on $K(I_1,I_2)$.
        \end{enumerate}
        Then $S_{\rm KKT}$ is isolated calm at $\bar{p}$ for $ (\bar{x},\bar{\lambda})$.
\end{theorem}

\begin{proof}{Proof.}
   From Lemma \ref{lem-iso}, we aim to prove that $(x,\lambda) = (\bar{x},\bar{\lambda})$ is an isolated point of $\L (\bu,\bv)$.
    We begin by demonstrating that $\L (\bu,\bv)$ has an isolated point $(\bar{x},\bar{\lambda})$ if the following generalized equation system
    \begin{equation}\label{isolated calmness}
       \left\{\begin{array}{ll}
        \begin{pmatrix}
            \nabla^2_{x^1x^1}L^1 & \cdots & \nabla^2_{x^1 x^N}L^1\\
            \vdots & \cdots & \vdots\\
            \nabla^2_{x^Nx^1}L^N & \cdots & \nabla^2_{x^N x^N}L^N
        \end{pmatrix}
        \begin{pmatrix}
            y^1\\
            \vdots\\
            y^N\\
            \end{pmatrix} + 
            
            \begin{pmatrix}
                (\mathcal{J}_{x^1} G^1)^{\mathrm{T}} & \cdots & 0\\
                 \vdots & \ddots & \vdots\\
                0  & \cdots & (\mathcal{J}_{x^N} G^N)^{\mathrm{T}}
            \end{pmatrix}
            \begin{pmatrix}
            \Delta \lambda^1\\
            \vdots\\
            \Delta \lambda^N\\
            \end{pmatrix} = 0,\\
    
        \nabla_{x^k} g^k_i (\bar{x}^k;\bar{w}^k)^{\mathrm{T}} y^k =0, \  \quad {\rm when}\ k=1,...,N, \ i\in I^k_1,\\
        0\le \Delta \lambda^k_i \perp -\nabla_{x^k} g^k_i (\bar{x}^k;\bar{w}^k)^{\mathrm{T}} y^k \ge 0 \quad {\rm when}\ k=1,...,N, \ i\in I^k_2,\\
        \Delta \lambda^k_ i =0 \quad {\rm when}\ k=1,...,N,\ i\in I^k_3
       \end{array}\right.
    \end{equation}
    at the given point $(\bp,\bx,\blam)\in \gph\, \S$ has $(y,\Delta \lambda) = (0,0)$ as an isolated solution. Assuming the contrary, let $(x_l, \lambda_l)$ be a sequence of solutions to (\ref{NEP-LKKT-local-calmness}) that is sufficiently close to $(\bar{x},\bar{\lambda})$. Since 
    $$
        -\bar{u}^k_i+ g^k_i (\bar{x}^k;\bar{w}^k)<0 \quad  {\rm for}\ {\rm  all}\ i\in I^k_3,
    $$
    it follows that $\lambda_{li}^k = 0$ for all $i\in I^k_3$ when $\lambda_l$ is sufficiently close to $\bar{\lambda}$, where $\lambda^k_{li}$ denotes the $i$-th component of player $k$'s Lagrange multiplier vector associated with $x_l$. Moreover,
    \begin{equation}\label{thm-isolated-calm-eq1}
        -\bar{u}^k_i+ g^k_i (\bar{x}^k;\bar{w}^k)+ \nabla_{x^k} g^k_i (\bar{x}^k;\bar{w}^k)^{\mathrm{T}} (x_l^k - \bar{x}^k) \in N_{\mathbb{R}_+}(\lambda_{li}^k),
    \end{equation}
     where $x^k_l$ denotes the player $k$'s strategy associated with $x_l$. Since  
     $$
     -\bar{u}^k_i+ g^k_i (\bar{x}^k;\bar{w}^k) =0 \quad  {\rm for}\ {\rm all}\ i\in I^k_1 \bigcup I^k_2,
     $$
     combined with \eqref{thm-isolated-calm-eq1}, we can obtain that  $\nabla_{x^k} g^k_i (\bar{x}^k;\bar{w}^k)^{\mathrm{T}} (x^k_l - \bx^k) =0$ for all $i\in I^k_1$, and
     $$
     0 \le \lambda^k_{li} \perp -\nabla_{x^k} g^k_i (\bar{x}^k;\bar{w}^k)^{\mathrm{T}} (x_l^k - \bar{x}^k) \ge 0 \quad {\rm for}\ {\rm all}\ i\in I^k_2.
     $$
     Consequently, $(y_l, \Delta \lambda_l) = (x_l-\bar{x}, \lambda_l- \bar{\lambda})$ is a solution to the generalized equation system (\ref{isolated calmness}) at $(\bp,\bx,\blam)\in \gph\, \S$, and it can be sufficiently close to $(0,0)$, which leads to a contradiction. We now prove that (\ref{isolated calmness}) has $(y,\Delta \lambda) = (0,0)$ as an isolated solution. We assume in contradiction that there exists a solution $(y^\prime, \Delta \lambda^\prime)$ close enough to $(0,0)$. If $y^\prime \neq 0$, multiplying the first equality of (\ref{isolated calmness}) by  
     $$
     \begin{pmatrix}
        (y^{\prime1})^{\mathrm{T}}\\
        \vdots\\
        (y^{\prime N})^{\mathrm{T}} 
     \end{pmatrix},
     $$
    we then obtain that $0\neq y^\prime \in K(I_1, I_2)$, and
    $$
    \sum\limits_{i=1}^N (y^{\prime k})^{\mathrm{T}} \nabla^2_{x^kx^i}L^k(\bar{x}^k,\bar{x}^{-k},\bar{\lambda}^k;\bar{w}^k) y^{\prime i} =0 \ {\rm for}\ k=1,...,N,
    $$
    which is a contradiction to the I-property at $(\bar{p},\bar{x},\bar{\lambda})$. Consequently, $y^\prime = 0$ and $\Delta \lambda^\prime \neq 0$. Then when $\Delta \lambda^\prime$ is sufficiently close to $0$, $\bar{\lambda}+\Delta \lambda^\prime$ is a Lagrange multiplier associated with $\bar{x}$ at $\bar{p}$, which is also a contradiction to the SMFCQ. \Halmos
    \end{proof}

Similar to strong regularity and continuously differentiable single-valued localization, investigating the isolated calmness of $S_{\rm KKT}$ is equivalent to studying the isolated calmness of $\T$.

\begin{corollary}\label{coro-iso-til}
    Let $ (\Bar{x}^1,\Bar{x}^2,...,\Bar{x}^N, \Bar{\lambda}^1,\Bar{\lambda}^2,...,\Bar{\lambda}^N)\in S_{\rm KKT} (\bar{p})$. The following are equivalent:
    \begin{enumerate}[i)]
        \item $S_{\rm KKT}$ is  isolated calm at $\bar{p}$ for $(\bar{x},\bar{\lambda})$;
        \item $\T$ is  isolated calm at $(\bar{u},\bar{v})$ for $(\bar{x},\bar{\lambda})$.
    \end{enumerate}
\end{corollary}

\begin{proof}{Proof.}
    This is an immediate result from Lemma \ref{lem-iso}. \Halmos
\end{proof}

The robustness of $S_{\rm KKT}$ is relatively challenging to acquire. As to some extent, robustness needs the investigation of the existence of NEs in a neighborhood of $\bar{x}$. In NLP, however, things become straightforward. For an isolated local minimizer, a perturbed local minimizer certainly exists, and it can be arbitrarily approximated to the initial minimizer before perturbation (\citet[Lemma 2.5]{dontchev2020characterizations}). However, this does not hold for an NEP. Consider the following example.

\begin{example}\label{robustness}
    \begin{equation}
        \begin{array}{lll}
       & {\rm Player}\ 1\quad
        \left\{\begin{array}{ll}
          \min\limits_x & \frac{1}{3}x^3-2xy+x+\varepsilon x\\
          {\rm s.\ t.} & x\le 1,
        \end{array}\right.\qquad \qquad
       & {\rm Player}\ 2\quad
        \begin{array}{ll}
          \min\limits_y & \frac{1}{2}y^2-xy.
        \end{array}
        \end{array}
        \end{equation}
\end{example}

When $\varepsilon =0$, it is clear that $(\bar{x},\bar{y})= (1,1)$ is an isolated LNE with $(\bar{x},\bar{y},\bar{\lambda})=(1,1,0)$ being an isolated solution to the KKT system
\begin{equation}
    \left\{\begin{array}{ll}
        x^2 - 2y +1 +\lambda= 0,\\
        0\le \lambda \perp 1-x \ge 0,\\
        y - x = 0.\\
    \end{array}\right.
\end{equation}
However, for any $\varepsilon>0$, the generalized equation system
\begin{equation}
    \left\{\begin{array}{ll}
        x^2 - 2y +1 = -\varepsilon-\lambda,\\
        0\le \lambda \perp 1-x \ge 0,\\
        y - x = 0\\
    \end{array}\right.
\end{equation}
has no solutions, which implies that the NE fails to exist. Consequently, we need additional assumptions to obtain the existence. 

\begin{definition}\label{def-conv}
    We say that Problem (\ref{NEP_perturbations}) satisfies the convex assumptions if 
    \begin{enumerate}[i)]
        \item $f^k  (\cdot,x^{-k};w^k)$ is  convex for any fixed $w^k$ and $x^{-k}$, $k=1,...,N$;
		\item $g^k_i  (\cdot;w^k)$ is an  affine function for any fixed $w^k$, $i=1,...,s_k$, and $k=1,...,N$;
		\item $g^k_j  (\cdot; w^k)$ is  convex for any fixed $w^k$, $j=s_k+1,...,m_k$, and $k=1,...,N$.
    \end{enumerate}
\end{definition}

\indent It is clear that the convex assumptions are not strict in NEPs. The proof of existence of NEs and the stability necessitate the convex assumptions or even compactness (\citet{facchinei2003finite,nash2024non,palomar2010convex}). Let $\mathcal{F}^k (p^k)$ denote the feasible region of each player's strategy at the perturbation $p^k$, and let $\mathcal{F}(p): = \mathcal{F}^1(p^1)\times \cdots \times  \mathcal{F}^N(p^N)$ denote the consolidation of all feasible regions. Recall that the mapping $E(p)$ denotes the NEs of Problem (\ref{NEP_perturbations}) at $p$. If $(\bar{x},\bar{\lambda})\in S_{\rm KKT}(\bar{p})$ is an isolated point, let $V_1 \times W_1$ be the neighborhood such that $S_{\rm KKT}(\bar{p})\cap V_1\times W_1 = \left\{(\bar{x},\bar{\lambda})\right\}$. Since $(\bu,\bv)$ is a fixed perturbation parameter, and any such fixed value does not affect the discussions of the stability properties, for the sake of simplicity, we may assume that $(\bu,\bv)=(0,0)$. The P-property implies that for all $y\in K(I_1,I_2)$ and $\Vert y \Vert = 1$, there exists a lower bound $\tau>0$ such that 
$$
\max\limits_{k=1,...,N}
\sum\limits_{i\neq k}^N (y^k)^{\mathrm{T}} \nabla^2_{x^kx^i}L^k y^i + \frac{1}{2}(y^k)^{\mathrm{T}} \nabla^2_{x^k x^k} L^k y^k> \tau>0 
$$
from the continuity.  

Now we can establish the following proposition.

\begin{proposition}\label{prop-deg-P}
    Let $(\Bar{x}^1,\Bar{x}^2,...,\Bar{x}^N, \Bar{\lambda}^1,\Bar{\lambda}^2,...,\Bar{\lambda}^N)\in S_{\rm KKT}(\bar{p})$. Suppose the following requirements are fulfilled:
        \begin{enumerate}[i)]
            \item Problem (\ref{NEP_perturbations}) satisfies the convex assumptions;
            \item The SMFCQ holds at $(\bar{p}^k,\bar{x}^k,\bar{\lambda}^k)$ for each player $k$;
            \item Problem (\ref{NEP_perturbations}) satisfies {\rm the P-property} at $(\bar{p},\bar{x},\bar{\lambda})$ on $K(I_1,I_2)$.
        \end{enumerate}
         Then there exists a neighborhood $V\subset V_1$ and $\bar{x}\in V$ such that 
    $$
    {\rm deg}\, (H_1(\cdot), V, 0)=1,
    $$
    where $H_t(z): [0,1]\times \mathbb{R}^n \rightarrow \mathbb{R}^n$,
    $$
    H_t(z) = z-  \argmin\limits_{x\in \mathcal{F}(\bar{p})} \left[ t \sum\limits_{k=1}^N  f^k(x^k,z^{-k}; \bar{w}^k) + \tau\Vert x-tz-(1-t)\bar{x}\Vert^2   \right].
    $$
\end{proposition}

\begin{proof}{Proof.}
    We begin by demonstrating that $\bar{x}$ is an isolated NE. Since the convex assumptions hold, for any $0 \neq y\in K(I_1, I_2)$, the P-property implies that 
    $$
        \max\limits_{k=1,...,N} \sum\limits_{i=1}^N (y^k)^{\mathrm{T}} \nabla^2_{x^k x^i} L^k y^i \ge 
        \max\limits_{k=1,...,N}
        \sum\limits_{i\neq k}^N (y^k)^{\mathrm{T}} \nabla^2_{x^kx^i}L^k y^i + \frac{1}{2}(y^k)^{\mathrm{T}} \nabla^2_{x^k x^k} L^k y^k > 0.
    $$
    Thus the I-property holds, and Theorem \ref{thm-iso}  indicates that $(\bar{x},\bar{\lambda})$ is isolated in the set $S_{\rm KKT}(\bar{p})$. Assume for contradiction that $\bar{x}$ is not an isolated NE. Then there exists a sequence $\{x_l\}$ with $x_l\in E(\bar{p}) \rightarrow \bar{x}$ when $l\rightarrow \infty$. 
    Since the SMFCQ holds at $(\bar{p}^k,\bar{x}^k,\bar{\lambda}^k)$ for each player $k$, \citet[Theorem 2.3]{robinson2009generalized} implies that the Lagrange multiplier sequence $\{\lambda_l\}$ associated with $\{x_l\}$ exists and is uniformly bounded for sufficiently large $l$.   Thus $x_l\in X_{\rm KKT}(\bar{p})$ if and only if $x_l \in E(\bar{p})$. Recall that the SMFCQ implies that the Lagrange multiplier $\lambda^k$ associated with $\bar{x}^k$ is unique for $k=1,...,N$ (see \citet{kyparisis1985uniqueness}). If there exists $M>0$ such that $\Vert \lambda_l \Vert\le M$ when $l$ is sufficiently large, let $\hat{\lambda}$ be a limit, and then $(\bar{x},\hat{\lambda})\in S_{\rm KKT}(\bar{p})$. The SMFCQ implies that $\hat{\lambda} = \bar{\lambda}$, which is a contradiction to the isolation of $(\bar{x},\bar{\lambda})$. 
    Consequently, $\bar{x}$ is an isolated NE.
    
    It follows from \citet[Proposition 12.5]{palomar2010convex} that $\bar{x}\in E(\bar{p})$ if and only if $H_1(\bar{x}) = 0$.  Berge's theorem implies that $H_t(z)\in C([0,1]\times \mathbb{R}^n)$. We now demonstrate that there exists a neighborhood $V$ with $\bar{x}\in V\subset V_1$ such that ${\rm deg}\, (H_1(\cdot), V, 0) = 1$. Assume for contradiction that this is false. Then there exist sequences $\left\{ t_l \right\}$ with $t_l \in [0,1]$ and  $\left\{ z_l \right\}$ with $ z_l\in\mathbb{R}^n$ such that $\bar{x} \neq z_l \rightarrow \bar{x}$ and $H_{t_l} (z_l)= 0$.  Moreover, $0<t_l<1$ when $l$ is sufficiently large. Without loss of generality, we assume that $t_l\in (0,1)$ for all $l$. Then 
    $$
        z_l = \argmin\limits_{x\in \mathcal{F}(\bar{p})} \left[ t_l \sum\limits_{k=1}^N  f^k(x^k,z_l^{-k}; \bar{w}^k) + \tau\Vert x-t_lz_l-(1-t_l)\bar{x}\Vert^2   \right],
    $$
    which implies that for all $x\in \mathcal{F}(\bar{p})$,
    $$
        (x^k-z_l^k)^{\mathrm{T}} \left[t_l \nabla_{x^k}f^k (z^k_l, z_l^{-k};\bar{w}^k)+ 2 \tau (1-t_l)(z_l^k-\bar{x}^k) \right]\ge 0.
    $$
    This yields 
    \begin{equation}\label{con-inq}
        (\bar{x}^k-z_l^k)^{\mathrm{T}} \nabla_{x^k}f^k (z_l^k, z_l^{-k};\bar{w}^k)\ge 0.
    \end{equation}
    Since $\bar{x}$ is an NE, 
    \begin{equation}\label{mini-inq}
        (z_l^k-\bar{x}^k)^{\mathrm{T}} \nabla_{x^k}f^k (\bar{x}^k, \bar{x}^{-k};\bar{w}^k) \ge 0.
    \end{equation}
    Dividing (\ref{con-inq}) and (\ref{mini-inq}) by $\Vert z_l^k - \bar{x}^k \Vert$ and passing $l\rightarrow \infty$, there exists 
    $$
        y : = \lim\limits_{l\rightarrow \infty} \frac{z_l -\bar{x}}{\Vert z_l -\bar{x} \Vert}
    $$
    such that $\nabla_{x^k}f^k (\bar{x}^k, \bar{x}^{-k};\bar{w}^k)^{\mathrm{T}} y^{ k} = 0$, which with the convex assumptions implies that $y\in K(I_1, I_2)$. Moreover, 
    \begin{equation}\label{f>=f}
        t_l f^k (\bar{x}^k,z^{-k}_l;\bar{w}^k) + \tau t_l \Vert z^{k}_l -\bar{x}^k \Vert^2 \ge t_l f^k (z^k_l,z^{-k}_l;\bar{w}^k) + \tau(1-t_l) \Vert z^{k}_l -\bar{x}^k \Vert^2,
    \end{equation}
    and 
    \begin{equation}\label{mutip}
        t_l \bar{\lambda}^k_i g^k_i (\bar{x}^k;\bar{w}^k) = 0 \ge 
        t_l \bar{\lambda}^k_i g^k_i(z^k_l ; \bar{w}^k) \quad {\rm for}\  i=1,...,m_k,\ k=1,...,N.
    \end{equation}
    Thus, combining (\ref{f>=f}) and (\ref{mutip}),
    \begin{equation}\label{L^k}
        \begin{array}{ll}
            & t_l \left( L^k(z^k_l,z^{-k}_{l},\bar{\lambda}^k;\bar{w}^k) - L^k(\bar{x}^k,z^{-k}_{l},\bar{\lambda}^k;\bar{w}^k) \right)+\tau (1-2t_l) \Vert z^k_l - \bar{x}^k \Vert^2\\
            = & t_l \nabla_{x^k} L^k(\bar{x}^k,z^{-k}_{l},\bar{\lambda}^k;\bar{w}^k)^{\mathrm{T}} (z^k_l - \bar{x}^k)  + \frac{1}{2} t_l (z^k_l - \bar{x}^k)^{\mathrm{T}} \nabla_{x^k x^k}L^k (\bar{x}^k,z^{-k}_l,\bar{\lambda}^k; \bar{w}^k) (z^k_l - \bar{x}^k)\\
            & +  \tau (1-2t_l) \Vert z^k_l - \bar{x}^k \Vert^2+ o(t_l \Vert z^k_l - \bar{x}^k \Vert^2)\\
            \le & 0.
        \end{array}
    \end{equation}
    Moreover,
    \begin{equation}\label{nabla L^k}
        t_l \nabla_{x^k} L^k (\bar{x}^k,\bar{x}^{-k},\bar{\lambda}^k;\bar{w}^k) = 0.
    \end{equation}
    From (\ref{L^k}) and (\ref{nabla L^k}), 
    \begin{equation}\label{t_l}
        \begin{array}{ll}
            &t_l \left( \nabla_{x^k} L^k(\bar{x}^k,z^{-k}_{l},\bar{\lambda}^k;\bar{w}^k) - \nabla_{x^k} L^k (\bar{x}^k,\bar{x}^{-k},\bar{\lambda}^k;\bar{w}^k) \right)^{\mathrm{T}} (z^k_l - \bar{x}^k)+ o( t_l\Vert z^k_l - \bar{x}^k \Vert^2)\\
            & + \frac{1}{2} t_l (z^k_l - \bar{x}^k)^{\mathrm{T}} \nabla_{x^k x^k}L^k (\bar{x}^k,z^{-k}_l,\bar{\lambda}^k; \bar{w}^k) (z^k_l - \bar{x}^k) +  \tau (1-2t_l) \Vert z^k_l - \bar{x}^k \Vert^2\\
            = & t_l \left( \sum\limits_{i\neq k}^N (z^k_l -\bar{x}^k)^{\mathrm{T}}\nabla^2_{x^k x^i}L^k (z^i_l - \bar{x}^i) + \frac{1}{2} (z^k_l -\bar{x}^k)^{\mathrm{T}}\nabla^2_{x^k x^k} L^k  (\bar{x}^k,z^{-k}_l,\bar{\lambda}^k; \bar{w}^k) (z^k_l - \bar{x}^k) \right) \\
            & +  \tau (1-2t_l) \Vert z^k_l - \bar{x}^k \Vert^2+ o( t_l\Vert z^k_l - \bar{x}^k \Vert^2)\\
            \le & 0.
        \end{array}
    \end{equation}
    Dividing (\ref{t_l}) by $t_l\Vert z_l - \bar{x} \Vert^2$ and passing $l\rightarrow \infty$, then we obtain the limit $y\in K(I_1,I_2)$ with $\Vert y \Vert =1$ such that
    $$
    \max\limits_{k=1,...,N}
    \sum\limits_{i\neq k}^N (y^k)^{\mathrm{T}} \nabla^2_{x^kx^i}L^k(\bar{x}^k,\bar{x}^{-k},\bar{\lambda}^k;\bar{w}^k) y^i + \frac{1}{2}(y^k)^{\mathrm{T}} \nabla^2_{x^k x^k} L^k (\bar{x}^k,\bar{x}^{-k},\bar{\lambda}^k;\bar{w}^k) y^k- \tau \le 0,
    $$
    which contradicts {\rm the P-property} at $(\bar{p},\bar{x},\bar{\lambda})$ on $K(I_1,I_2)$. Thus there exists a neighborhood $V\subset V_1$ such that ${\rm deg}\, (H_1(\cdot), V, 0) = 1$. \Halmos 

\end{proof}

By Proposition \ref{prop-deg-P}, we obtain the following theorem, which provides an exact characterization of the robust isolated calmness of $S_{\rm KKT}$.

\begin{theorem}\label{thm-robust}
    Let $(\Bar{x}^1,\Bar{x}^2,...,\Bar{x}^N, \Bar{\lambda}^1,\Bar{\lambda}^2,...,\Bar{\lambda}^N)\in S_{\rm KKT}(\bar{p})$. Suppose the following requirements are fulfilled: 
        \begin{enumerate}[i)]
            \item Problem (\ref{NEP_perturbations}) satisfies the convex assumptions;
            \item The SMFCQ holds at $(\bar{p}^k,\bar{x}^k,\bar{\lambda}^k)$ for each player $k$;
            \item Problem (\ref{NEP_perturbations}) satisfies {\rm the P-property} at $(\bar{p},\bar{x},\bar{\lambda})$ on $K(I_1,I_2)$.
        \end{enumerate}
        Then $S_{\rm KKT}$ is robust isolated calm at $\bar{p}$ for $(\bar{x},\bar{\lambda})$, and for any $(p,x(p),\lambda(p))\in {\rm gph}\, S_{\rm KKT}$ around $(\bar{p},\bar{x},\bar{\lambda})$, $x(p)$ is an NE of Problem (\ref{NEP_perturbations}).
\end{theorem}

\begin{proof}{Proof.}
    Consider the following mapping 
    $$
    \hat{H}(z;p) = z-\argmin\limits_{x\in \mathcal{F}(p)} \left[  \sum\limits_{k=1}^N f^k(x^k,z^{-k}; w^k)-\left<v,x\right> + \tau\Vert x-z\Vert^2   \right].
    $$
    From Theorem \ref{thm-iso}, there exist neighborhoods $U_1$ of $\bar{p}$, $V_1\times W_1$ of $(\bar{x},\bar{\lambda})$ and a constant $\kappa \ge 0$ such that 
    $$
        S_{\rm KKT}(p)\cap V_1\times W_1 \subset \left\{(\bar{x},\bar{\lambda})\right\} + \kappa \Vert p-\bar{p}\Vert \mathbb{B} \quad {\rm for}\ {\rm all}\ p\in U_1.
    $$
    Without loss of generality, we may assume that $S_{\rm KKT}(\bar{p})\cap {\rm cl\,}V_1\times {\rm cl\,}W_1 = \left\{(\bar{x},\bar{\lambda})\right\}$ and $E(\bar{p})\cap {\rm cl\,}V_1 = \left\{\bar{x}\right\}$ from Proposition \ref{prop-deg-P}. We begin by proving that there exist neighborhoods $ V\subset V_1$ of $\bar{x}$ and $ U_2\subset U_1$ of $\bar{p}$ such that $E(p)\cap V \neq \emptyset$ for any $p\in U_2$.
    It is clear that $z-\hat{H}(z;p)$ is actually the unique global minimizer of the following convex optimization problem
    $$
        {\rm \textbf{H}(z;p)} \quad \quad
        \begin{array}{ll}
         \min\limits_{x} &  \sum\limits_{k=1}^N  f^k(x^k,z^{-k}; w^k) -\left<v,x\right> + \tau\Vert x-z\Vert^2\\
       {\rm s.\ t.}& g_i^k (x^k;w^k) = u_i^k, \  i=1,...,s_k,\ k=1,...,N,\\
       & g^k_j  (x^k;w^k) \le u_j^k,\ j=s_k+1,...,m_k,\ k=1,...,N.\\
       \end{array}
    $$
    Since the SMFCQ holds at $(\bar{z},\bar{p},\bar{x},\bar{\lambda})= (\bar{x},\bar{p},\bar{x}, \blam)$ for ${\rm \textbf{H}(z;p)}$, it follows from \citet[Lemma 2.5]{dontchev2020characterizations} that $\hat{H}(\cdot;\cdot)$ is lower semi-continuous (as a set-valued mapping) at $(\bar{x}, \bar{p})$ for $0$. Noticing that $\hat{H}(\cdot;\cdot)$ is a single-valued mapping for the sake of strict convexity, then $\hat{H}(\cdot;\cdot)$ is continuous at $(\bar{x}, \bar{p})$. For $(z,p)$ close enough to $(\bar{x},\bar{p})$,
    $$
        \begin{array}{ll}
        &\Vert z-\hat{H}(z;p) - \bar{x}\Vert\\
         = & \Vert z-\bar{x}+(\hat{H}(\bar{x};\bar{p})-\hat{H}(z;p)) \Vert\\
         \le & \Vert z-\bar{x} \Vert + \Vert \hat{H}(\bar{x};\bar{p})-\hat{H}(z;p) \Vert.
        \end{array}
    $$
    Thus, $(z-\hat{H}(z;p), p)$ is also close to $(\bar{x},\bar{p})$. By \citet[Theorem 2.3]{robinson2009generalized} and \citet{kyparisis1985uniqueness}, the MFCQ holds at $(z,p,z-\hat{H}(z;p))$ for $\textbf{H}(z;p)$ when $(z,p)$ is close enough to $(\bar{x},\bar{p})$. Then from \citet[Lemma 2.5]{dontchev2020characterizations}, $\hat{H}(\cdot;\cdot)$ is continuous in a small neighborhood of $(\bar{x},\bar{p})$. Without loss of generality, let $\hat{H}(\cdot ;\cdot)$ be continuous in ${\rm cl\,} V_1 \times {\rm cl\,} U_1$.
      Let $V\subset V_1$ be the neighborhood presented by Proposition \ref{prop-deg-P}. Then there exists a neighborhood $U_2 \subset U_1$ such that $\hat{H}(\cdot ;\cdot)$ is uniformly continuous in ${\rm cl\,} V \times {\rm cl\,} U_2$, and 
    $$
    \max\limits_{z\in {\rm cl\,} V}\Vert \hat{H}(z,\bar{p})-\hat{H}(z,p) \Vert_\infty < \inf\limits_{z\in H_1({\rm bd\,}V)} \Vert z \Vert_\infty \quad {\rm for}\ {\rm all}\ p\in U_2.
    $$
     Following from Lemma \ref{lem-deg},
    $$
    1 = {\rm deg}\, (H_1(\cdot), V, 0) = {\rm deg}\, (\hat{H}(\cdot,p), V ,0) \quad {\rm for}\ {\rm all}\ p\in U_2.
    $$
    Consequently, from Lemma \ref{lem-deg}, $E(p)\cap V \neq \emptyset$ for any $p\in U_2$.

    We now prove that there exist neighborhoods $W\subset W_1$ of $\bar{\lambda}$ and $U\subset U_2$ of $\bar{p}$ such that $S_{\rm KKT}(p)\cap V\times W \neq \emptyset$ for all $p\in U$. Since the SMFCQ holds at $(\bar{p}^k,\bar{x}^k,\bar{\lambda}^k)$, for any $p$ in a small neighborhood of $\bar{p}$, $\lambda(p)$ exists and is uniformly bounded. Assume for contradiction that there exist constant $\alpha>0$ and $p_l\rightarrow \bar{p}$ such that $\Vert \lambda(p_l)-\bar{\lambda} \Vert \ge \alpha $. Let
    $$
        \hat{x}: = \lim\limits_{l\rightarrow \infty} x(p_l),\ \hat{\lambda}: = \lim\limits_{l\rightarrow \infty} \lambda(p_l).
    $$
    Then $(\hat{x},\hat{\lambda})\in  S_{\rm KKT}(\bar{p})$, which implies that $\hat{x}\in E(\bar{p})$. If $\hat{x}= \bar{x}$, the SMFCQ yields that $\bar{\lambda} = \hat{\lambda}$. This is a contradiction as $\Vert \bar{\lambda}-\hat{\lambda}\Vert \ge \alpha$. Thus $\hat{x}\neq \bar{x}$ and $\hat{x}\in {\rm cl\,}V \subset {\rm cl\,}V_1$, which is also a contradiction since $E(\bar{p})\cap {\rm cl\,}V_1 = \left\{ \bar{x}\right\}$. Consequently, there exist neighborhoods $W\subset W_1$ of $\bar{\lambda}$ and $U\subset U_2$ of $\bar{p}$ such that $S_{\rm KKT}(p)\cap V\times W \neq \emptyset$ for any $p\in U$. Then
    $$
        S_{\rm KKT}(p)\cap V\times W \subset S_{\rm KKT}(p)\cap V_1\times W_1 \subset \left\{(\bar{x},\bar{\lambda})\right\} + \kappa \Vert p-\bar{p}\Vert \mathbb{B} \quad {\rm for}\ {\rm all}\ p\in U,
    $$
    which concludes the proof. \Halmos
\end{proof}

\begin{remark}\label{rmk-robust} 
        When $N=1$, i.e., the NEP reduces to a convex nonlinear programming problem, Theorem \ref{thm-robust} reduces to the condition that the SOSC combined with the SMFCQ implies the robust isolated calmness of $S^1_{\rm KKT}$ by Dontchev and Rockafellar for NLP \citep{dontchev1996characterizations}.
\end{remark}

It is remarkable that the robust isolated calmness of $\T$ can also imply the robust isolated calmness of $S_{\rm KKT}$.

\begin{corollary}\label{coro-rob-til}
    Let $ (\Bar{x}^1,\Bar{x}^2,...,\Bar{x}^N, \Bar{\lambda}^1,\Bar{\lambda}^2,...,\Bar{\lambda}^N)\in S_{\rm KKT} (\bar{p})$. The following are equivalent:
    \begin{enumerate}[i)]
        \item $S_{\rm KKT}$ is  robust isolated calm at $\bar{p}$ for $(\bar{x},\bar{\lambda})$;
        \item $\T$ is robust isolated calm at $(\bar{u},\bar{v})$ for $(\bar{x},\bar{\lambda})$.
    \end{enumerate}
\end{corollary}

\begin{proof}{Proof.}
    It suffices to demonstrate that $ii)\Longrightarrow i) $. Following from Corollary \ref{coro-iso-til}, $S_{\rm KKT}$ is isolated calm at $\bar{p}$ for $(\bar{x},\bar{\lambda})$. Let the constant $\kappa_1\ge 0$, the neighborhoods $U_1$ of $\bar{p}$ and $V_1\times W_1$ of $(\bar{x},\bar{\lambda})$ satisfy 
    $$
        S_{\rm KKT}(p)\cap V_1\times W_1 \subset \left\{(\bar{x},\bar{\lambda})\right\} + \kappa_1 \Vert p-\bar{p} \Vert \mathbb{B} \quad {\rm for}\ {\rm all}\ p\in U_1.
    $$
    Let the constant $\kappa_2 \ge 0$, the neighborhoods $U_2$ of $\bar{p}$ and $V_2\times W_2$ of $(\bar{x},\bar{\lambda})$ satisfy 
    $$
    \T(u,v)\cap V_2\times W_2 \subset \left\{(\bar{x},\bar{\lambda})\right\} + \kappa_2 (\Vert u-\bar{u}\Vert+\Vert v-\bar{v} \Vert )\mathbb{B} \quad {\rm for}\ {\rm all}\ (u,v,\bar{w})\in U_2,
    $$
    and 
    $$
    \T(u,v)\cap V_2\times W_2 \neq \emptyset \quad {\rm for}\ {\rm all}\ (u,v,\bar{w})\in U_2.
    $$
    It suffices to prove that there exist neighborhoods $U\subset U_1$ of $\bar{p}$ and $\kappa \ge \kappa_1$ such that 
    $$
    \emptyset \neq S_{\rm KKT}(p)\cap V_1\times W_1 \subset \left\{(\bar{x},\bar{\lambda})\right\} + \kappa \Vert p-\bar{p} \Vert \mathbb{B}  \quad {\rm for}\ {\rm all}\ p\in U.
    $$
    For any $(x,\lambda)\in V_1\times W_1$ and $(u,v,\bar{w})\in U_2$, let
    \begin{equation}\label{13}
        \begin{pmatrix}
            v^\prime\\
            -u^\prime
        \end{pmatrix}
        = 
        \begin{pmatrix}
            v\\
            -u
        \end{pmatrix}+
        \textbf{L}(x,\lambda;\bar{w})- \textbf{L}(x,\lambda;w).
    \end{equation}
    We choose the neighborhood $U\subset U_1$ of $\bar{p}$ to be sufficiently small such that there exists a constant $C>0$ satisfying
    $$
    \Vert (u^\prime,v^\prime)-(\bar{u},\bar{v}) \Vert \le \Vert u-\bar{u} \Vert + \Vert v-\bar{v}\Vert + \max_{(x,\lambda,p^\prime)\in {\rm cl\,} (V_1\times W_1 \times U)} \Vert \mathcal{J}_w \textbf{L} (x,\lambda;w^\prime) (w-\bar{w}) \Vert \le C \Vert p-\bar{p} \Vert.
    $$
    Then, for sufficiently small $U$, there exists $(x(u^\prime,v^\prime),\lambda(u^\prime,v^\prime))\in \T(u^\prime,v^\prime)$ such that
    $$
        \Vert (x(u^\prime,v^\prime),\lambda(u^\prime,v^\prime)) - (\bar{x},\bar{\lambda})\Vert \le 
        \kappa_2 \Vert (u^\prime,v^\prime)-(\bar{u},\bar{v}) \Vert \le \kappa_2 C \Vert p-\bar{p} \Vert.
    $$
    From (\ref{13}) and the definition of $\T$, we obtain $(x(u^\prime,v^\prime),\lambda(u^\prime,v^\prime))\in S_{\rm KKT}(p)$. Consequently, for sufficiently small $U$, let $\kappa = \max\left\{C\kappa_2,\kappa_1\right\}$, and then 
    $$
    (x(u^\prime,v^\prime),\lambda(u^\prime,v^\prime)) \in S_{\rm KKT} \cap V_1 \times W_1 \subset \left\{(\bar{x},\bar{\lambda}) \right\} + \kappa \Vert p -\bar{p} \Vert \mathbb{B} \quad {\rm for}\ {\rm all}\ p\in U,
    $$
    which concludes the proof. \Halmos
\end{proof}

\section{Player's Problems Being Quadratic Programs}\label{QP}

A crucial class of models is NEPs in which each player solves a QP problem (\citet{facchinei2003finite,young2014handbook}). In these models, each player is subject to individual linear constraints, and their payoff function is quadratic, interdependent on the strategies of other players. Utilizing the characterizations presented in Sections \ref{Strong}, \ref{Diff}, and \ref{Robust}, we conduct detailed stability analysis of a standard QP model in the context of NEP. 
Let $\mathbb{M}_{n\times n}$ be the space of $n\times n$ matrices over the field of real numbers. Consider the following model 

\begin{equation}\label{NEP_QP_perturbations}
    {\rm Player}\ k\quad \quad
    \begin{array}{ll}
     \min\limits_{x^k \in \mathbb{R}^{n_k}} & \frac{1}{2}x^{\mathrm{T}} P^k(w^k) x - \left<c^k(w^k),x\right>-\left< v^k , x^k\right>\\
   {\rm s.\ t.}& a_i^k(w^k)^{\mathrm{T}} x^k = b^k_i + u^k_i, \  i=1,...,s_k,\\
   & a^k_j  (w^k)^{\mathrm{T}}x^k \le b^k_j + u^k_j,\ j=s_k+1,...,m_k,\\
   \end{array}
   \end{equation}
    where $P^k(\cdot): \mathbb{R}^{d_k} \rightarrow \mathbb{M}_{n\times n}$, $c^k(\cdot): \mathbb{R}^{d_k}\rightarrow \mathbb{R}^{n}$ and $a^k_i(\cdot): \mathbb{R}^{d_k}\rightarrow \mathbb{R}^{n_k}$ are continuously differentiable around $\bar{w}^k$ for $k=1,..,N$, $i=1,...,m_k$. In the following discussions, since we only need the values of $P^k(\cdot)$, $c^k (\cdot)$, $a^k_i ( \cdot)$ at $\bw^k$, for simplicity, we define $P^k = P^k(\bw^k)$, $c^k = c^k (\bar{w}^k)$, and $a^k_i = a^k_i (\bar{w}^k)$ for $k=1,..,N$, $i=1,...,m_k$. Moreover, since the stability analysis at any fixed point $(\bu, \bv, \bw)\in \R^{m+n+d}$ is essentially the same as for $(\bu, \bv, \bw) = (0,0,0)$, we directly set $\bp = (\bu, \bv, \bw) = (0,0,0)$ to simplify the presentation. Let $(\Bar{x}^1,\Bar{x}^2,...,\Bar{x}^N, \Bar{\lambda}^1,\Bar{\lambda}^2,...,\Bar{\lambda}^N)\in S_{\rm KKT}(\bar{p})$. For  $k=1,...,N$, we define
$$
   A^k = 
   \begin{pmatrix}
    (a^k_1)^{\mathrm{T}}\\
    \vdots\\
    (a^k_{m_k})^{\mathrm{T}}
   \end{pmatrix}.
$$
The index sets \eqref{def-index-NEP} associated with Problem \eqref{NEP_QP_perturbations} at the given point $(\bp,\bx,\blam)\in \gph\, \S$ are
$$
\begin{array}{ll}
& I^k_1  = \left\{i \, | \, \bar{\lambda}^k_i > 0  = (a^k_i)^{\mathrm{T}} \bar{x}^k-b^k_i  \right\}\bigcup \left\{ 1,...,s_k\right\},\\
& I^k_2  = \left\{i \, | \, \bar{\lambda}^k_i = 0  =  (a^k_i)^{\mathrm{T}} \bar{x}^k-b^k_i  \right\},\\
& I^k_3  = \left\{i \, | \, \bar{\lambda}^k_i = 0  >  (a^k_i)^{\mathrm{T}} \bar{x}^k-b^k_i  \right\}.
\end{array}
$$
Let $A^k_1$ and $A^k_2$ denote the submatrices of $A^k$ corresponding to $I^k_1$ and $I^k_2$ for $k=1,...,N$. The index partition set $\mathcal{I}$ associated with Problem \eqref{NEP_QP_perturbations} is
$$
\mathcal{I} = \left\{ (J_1,J_2,J_3) 
    \, \bigg{|}  \,
            \begin{array}{ll}    
            & J_1 = \prod_{k=1}^N J^k_1, J_2 = \prod_{k=1}^N J^k_2, J_3 = \prod_{k=1}^N J^k_3,\ {\rm where}\ {\rm for}\ k=1,...,N,\\ 
            & J^k_1 \bigcup J^k_2 \bigcup J^k_3= \{1,...,m_k\} \ {\rm with}\
            I^k_1 \subset J^k_1 \subset I^k_1 \bigcup I^k_2, I^k_3 \subset J^k_3 \subset I^k_2\bigcup I^k_3\\
            \end{array}
    \right\}.
$$
Then for any partition $(J_1,J_2,J_3)\in \mathcal{I}$, 
$$
            K(J_1,J_2) = \left\{
                y=(y^1,...,y^N)\in \mathbb{R}^{n} \, \bigg{|}  \,
            \begin{array}{ll}    
            & (a^k_i)^{\mathrm{T}} y^k =0 \quad {\rm for}\ {\rm all}\ i\in J^k_1,\\ &(a^k_i)^{\mathrm{T}} y^k \le 0 \quad {\rm for}\ {\rm all}\ i\in J^k_2
            \end{array}\right\}.
        $$ 
        We also define the subspace 
$$
        M^k  = \left\{  y^k\in \mathbb{R}^{n_k} \mid  (a^k_i)^{\mathrm{T}} y^k =0 \quad {\rm for}\ {\rm all}\ i\in I^k_1\right\} \quad {\rm for}\ k=1,...,N.
$$
Suppose that $P^k$ can be written as the following symmetric block matrix
$$
    P^k = \begin{pmatrix}
        P^k_{n_1 n_1} & \cdots & P^k_{n_1 n_N}\\
        \vdots & \ddots & \vdots\\
        (P^k_{n_1 n_N})^{\mathrm{T}} & \cdots & P^k_{n_N n_N}
    \end{pmatrix},
    \quad k=1,...,N.  
$$

By Theorem \ref{thm-strong-suf}, we obtain the following proposition.

\begin{proposition}\label{pro-qp-strong}
    Let $(\Bar{x}^1,\Bar{x}^2,...,\Bar{x}^N, \Bar{\lambda}^1,\Bar{\lambda}^2,...,\Bar{\lambda}^N)\in S_{\rm KKT}(\bar{p})$. Suppose that the following requirements are fulfilled:
    \begin{enumerate}[i)]
        \item $\begin{pmatrix}
            A^k_1\\
            A^k_2
        \end{pmatrix}$ is of full row rank for $k=1,...,N$;
        \item For each player $k$ and any $0\neq y^k\in M^k$,
        $$
            (y^k)^{\mathrm{T}} P^k_{n_k n_k}y^k > 0;
        $$
        \item There exists a set of parameters 
        $$
            \left\{ \alpha_{i,j}>0 \, \bigg{|} \,  1\le i,j\le N,\, j\neq i,\, {\rm and}\  \sum\limits_{\substack{ j\neq i}} \alpha_{i,j} = 1\ {\rm for}\ i=1,...,N \right\}
        $$
        such that for any $0\neq y^k\in M^k$ and $T_{ki}: = P^k_{n_k n_i}+ (P^i_{n_i n_k})^{\mathrm{T}}$, 
        $$
            (y^k)^{\mathrm{T}} 
            \left(P^k_{n_k n_k} -\frac{1}{4 \alpha_{i,k}\alpha_{k,i}}\left(T_{ik}^{\mathrm{T}} B^i \left(({B^i})^{\mathrm{T}} P^i_{n_i n_i} B^i\right)^{-1}({B^i})^{\mathrm{T}} T_{ik}\right)\right)
            y^k>0
        $$
        for each player $i$ and player $k$, where $i\neq k$. 
    \end{enumerate}
    Then $S_{\rm KKT}(p)$ has a Lipschitz continuous single-valued localization around $\bar{p}$ for $(\bar{x},\bar{\lambda})$. Moreover, for any $(p,x(p),\lambda(p))\in {\rm gph}\, S_{\rm KKT}$ around $(\bar{p},\bar{x},\bar{\lambda})$, $x(p)$ is an LNE of Problem (\ref{NEP_QP_perturbations}).
\end{proposition}

In the following, we provide an intuitive example to illustrate Proposition \ref{pro-qp-strong}.

\begin{example}\label{exm-strong-regularity}
    Consider the following NEP:
    \begin{equation}
        \begin{array}{ll}
         {\rm Player}\ 1\quad
        & \left\{\begin{array}{ll}
          \min\limits_{x^1\in \mathbb{R}^2} & \frac{1}{2}(x^1_1,x^1_2,x^2_1)
          \begin{pmatrix}
            1 & 0 & 0\\
            0 & -1 & 1\\
            0 & 1 & 0
          \end{pmatrix}
          \begin{pmatrix}
            x^1_1\\
            x^1_2\\
            x^2_1
          \end{pmatrix}
          +\varepsilon x^1_1\\
          {\rm s.\ t.} & x^1_1\le 0, \\
          & x^1_2\le 0.
        \end{array}\right.\\
        ~\\
        {\rm Player}\ 2\quad
        & \begin{array}{ll}
          \min\limits_{x^2\in \mathbb{R}} & \frac{1}{2}(x^1_1,x^1_2,x^2_1)
          \begin{pmatrix}
            0 & 0 & -1\\
            0 & 0 & 0\\
            -1 & 0 & 1
          \end{pmatrix}
          \begin{pmatrix}
            x^1_1\\
            x^1_2\\
            x^2_1
          \end{pmatrix}+x^2_1.
        \end{array}
        \end{array}
        \end{equation}
\end{example}
When $\varepsilon =0$, we observe that 
$(\bar{x}^1_1,
        \bar{x}^1_2,
        \bar{x}^2_1,
        \bar{\lambda}^1_1,
        \bar{\lambda}^1_2)
     = 
     (0,0,-1,0,1) \in S_{\rm KKT}(0)
$
is an isolated KKT triple. Following from the notations of Proposition \ref{pro-qp-strong}, one has 
    $ M^1 = \mathbb{R}\times \left\{0\right\}, M^2 = \mathbb{R}$, 
    \begin{math}
        A^1_1 = \left( 
            \begin{smallmatrix}
                0 & 1
            \end{smallmatrix}
        \right)
    \end{math},
    \begin{math}
        A^1_2 = \left( 
            \begin{smallmatrix}
                1 & 0
            \end{smallmatrix}
        \right)
    \end{math},
    \begin{math}
        B^1 = \left( 
            \begin{smallmatrix}
                1 \\
                0
            \end{smallmatrix}
        \right)
    \end{math},
    and 
    \begin{math}
        B^2 = \left( 
            \begin{smallmatrix}
                1
            \end{smallmatrix}
        \right)
    \end{math}. 
Let $\alpha_{1,2} = \alpha_{2,1} = 1$. For player $1$, 
\begin{enumerate}[i)]
    \item $\begin{pmatrix}
        A^1_1\\
        A^1_2
    \end{pmatrix}$ is of full row rank;
    \item For any $0 \neq y^1 = (y^1_1,y^1_2)\in \mathbb{R}\times \left\{ 0\right\}$,
    $$
        (y^1)^{\mathrm{T}} P^1_{n_1 n_1} y^1 = (y^1_1)^2  >0;
    $$
    \item For any $0 \neq y^1 = (y^1_1,y^1_2)\in \mathbb{R}\times \left\{ 0\right\}$,
    \begin{math}
       T_{12} =\left(
            \begin{smallmatrix}
                -1\\
                1
            \end{smallmatrix}
        \right)
    \end{math},
    $$
    (y^1)^{\mathrm{T}} 
    \left(P^1_{n_1 n_1} -\frac{1}{4\alpha_{1,2}\alpha_{2,1}}\left(T_{12} B^2 \left(({B^2})^{\mathrm{T}} P^2_{n_2n_2} B^2\right)^{-1}({B^2})^{\mathrm{T}}(T_{12})^{\mathrm{T}}\right)\right)
    y^1 = \frac{3}{4}(y^1_1)^2>0.
    $$
\end{enumerate}
For player $2$,
\begin{enumerate}[i)]
    \item For any $0 \neq y^2 = (y^2_1)\in \mathbb{R}$,
    $$
        (y^2)^{\mathrm{T}} P^2_{n_2 n_2} y^2 = (y^2_1)^2  >0;
    $$
    \item For any $0 \neq y^2 = (y^2_1)\in \mathbb{R}$,
    \begin{math}
        T_{21} =\left(
             \begin{smallmatrix}
                 -1 & 1
             \end{smallmatrix}
         \right)
     \end{math},
    $$
    (y^2)^{\mathrm{T}} 
    \left(P^2_{n_2 n_2} -\frac{1}{4\alpha_{1,2}\alpha_{2,1}}\left(T_{21} B^1 \left(({B^1})^{\mathrm{T}} P^1_{n_1n_1} B^1\right)^{-1}({B^1})^{\mathrm{T}}(T_{21})^{\mathrm{T}}\right)\right)
    y^2 = \frac{3}{4}(y^2_1)^2>0.
    $$
\end{enumerate}
Following from Proposition \ref{pro-qp-strong}, $S_{\rm KKT}$ is strongly regular at $\bar{\varepsilon} = 0$ for $(\bar{x},\bar{\lambda}) = (0,0,-1,0,1)$. Moreover, it is clear that the SCSC fails to hold at $(\bar{\varepsilon},\bar{x},\bar{\lambda})=(0,0,0,0,1)$ for player 1, which yields that $S_{\rm KKT}$ does not have a continuously differentiable single-valued localization from Theorem \ref{thm-dif-eq}. In fact, we observe that for sufficiently small neighborhoods $V$ of $\bar{x}$ and $W$ of $\bar{\lambda}$,
$$
    S_{\rm KKT}(\varepsilon)\cap V\times W = 
    \left\{\begin{array}{ll}
        \left(-\varepsilon,0,-1-\varepsilon,0,\varepsilon+1\right)^{\mathrm{T}} & {\rm when}\ \varepsilon>0\  {\rm is}\ {\rm sufficiently}\ {\rm small},\\
        (0, 0 , -1,-\varepsilon,1)^{\mathrm{T}} & {\rm when}\ \varepsilon<0 \  {\rm is}\ {\rm sufficiently}\ {\rm small}.
    \end{array}\right.
$$

By Corollary \ref{coro-dif-suf}, we obtain the following proposition.

\begin{proposition}\label{prop-qp-dif}
    Let $(\Bar{x}^1,\Bar{x}^2,...,\Bar{x}^N, \Bar{\lambda}^1,\Bar{\lambda}^2,...,\Bar{\lambda}^N)\in S_{\rm KKT}(\bar{p})$. Suppose that the following requirements are fulfilled:
    \begin{enumerate}[i)]
        \item $I^k_2  = \emptyset$ for $k=1,...,N$;
        \item $A^k_1$ is of full row rank for $k=1,...,N$;
        \item For each player $k$ and any $0\neq y^k\in M^k$,
        $$
            (y^k)^{\mathrm{T}} P^k_{n_k n_k}y^k > 0;
        $$
        \item There exists a set of parameters 
        $$
            \left\{ \alpha_{i,j}>0 \, \bigg{|} \,  1\le i,j\le N,\, j\neq i,\, {\rm and}\  \sum\limits_{\substack{ j\neq i}} \alpha_{i,j} = 1\ {\rm for}\ i=1,...,N \right\}
        $$
        such that for any $0\neq y^k\in M^k$ and $T_{ki}: = P^k_{n_k n_i}+ (P^i_{n_i n_k})^{\mathrm{T}}$, 
        $$
            (y^k)^{\mathrm{T}} 
            \left(P^k_{n_k n_k} -\frac{1}{4 \alpha_{i,k}\alpha_{k,i}}\left(T_{ik}^{\mathrm{T}} B^i \left(({B^i})^{\mathrm{T}} P^i_{n_i n_i} B^i\right)^{-1}({B^i})^{\mathrm{T}} T_{ik}\right)\right)
            y^k>0
        $$
        for each player $i$ and player $k$, where $i\neq k$. 
    \end{enumerate}
    Then $S_{\rm KKT}(p)$ has a continuously differentiable single-valued localization around $\bar{p}$ for $(\bar{x},\bar{\lambda})$. Moreover, for any $(p,x(p),\lambda(p))\in {\rm gph}\, S_{\rm KKT}$ around $(\bar{p},\bar{x},\bar{\lambda})$, $x(p)$ is an LNE of Problem (\ref{NEP_QP_perturbations}).
\end{proposition}

By Theorem \ref{thm-robust}, we obtain the following proposition.

\begin{proposition}\label{prop-qp-robust}
    Let $(\Bar{x}^1,\Bar{x}^2,...,\Bar{x}^N, \Bar{\lambda}^1,\Bar{\lambda}^2,...,\Bar{\lambda}^N)\in S_{\rm KKT}(\bar{p})$. For each player $k$, suppose the following requirements are fulfilled:
        \begin{enumerate}[i)]
            \item $P^k_{n_k n_k}$ is positive semi-definite;
            \item $A_1^k$ is of full row rank, and there exists $z^k\in \mathbb{R}^{n_k}$ such that 
            $$
                A^k_1 z^k = 0 \ {\rm and}\ A^k_2 z^k < 0;
            $$
            \item For any $0\neq y\in  K(I_1,I_2)$,
            $$
                \max\limits_{k=1,...,N}
                \sum\limits_{i\neq k}^N (y^k)^{\mathrm{T}} P^k_{n_k n_i} y^i + \frac{1}{2}(y^k)^{\mathrm{T}} P^k_{n_k n_k}y^k >0. 
            $$
        \end{enumerate}
        Then $S_{\rm KKT}$ is robust isolated calm at $\bar{p}$ for $(\bar{x},\bar{\lambda})$, and for any $(p,x(p),\lambda(p))\in {\rm gph}\, S_{\rm KKT}$ around $(\bar{p},\bar{x},\bar{\lambda})$, $x(p)$ is an NE of Problem (\ref{NEP_perturbations}).
\end{proposition}

In the following, we provide an intuitive example to illustrate Proposition \ref{prop-qp-robust}.

\begin{example}\label{exm1-robust-isolated-calmness}
    Consider the following  NEP:
    \begin{equation}
        \begin{array}{ll}
         {\rm Player}\ 1\quad
        & \left\{\begin{array}{ll}
          \min\limits_{x^1\in \mathbb{R}} & \frac{1}{2}(x^1_1,x^2_1)
          \begin{pmatrix}
            1 & 1\\
            1 & 0
          \end{pmatrix}
          \begin{pmatrix}
            x^1_1\\
            x^2_1
          \end{pmatrix}
         + \varepsilon x^1_1\\
          {\rm s.\ t.} 
          & x^1_1\le 2\e,
        \end{array}\right.\\
        ~\\
        {\rm Player}\ 2\quad
        & \left\{\begin{array}{ll}
          \min\limits_{x^2\in \mathbb{R}} & \frac{1}{2}(x^1_1,x^2_1)
          \begin{pmatrix}
            0 & 2\\
            2 & 1
          \end{pmatrix}
          \begin{pmatrix}
            x^1_1\\
            x^2_1
          \end{pmatrix}
         +2 \varepsilon x^2_1\\
          {\rm s.\ t.} 
          & x^2_1\le \e.
        \end{array}\right.
        \end{array}
        \end{equation}
\end{example}
When $\varepsilon =0$, we observe that $(\bar{x}^1_1,
\bar{x}^2_1,
\bar{\lambda}^1_1,
\bar{\lambda}^2_1)
= 
(0,0,0,0) \in S_{\rm KKT}(0)
$
is an isolated KKT triple. 
Following from the notations of Proposition \ref{pro-qp-strong}, 
let $z^1 = -1$, $z^2 = -2$. Then for
\begin{math}
    A^1_2 = \left(\begin{smallmatrix}
        1
     \end{smallmatrix}\right),
     A^2_2 = \left(\begin{smallmatrix}
        1
     \end{smallmatrix}\right),
\end{math}
$A^1_2 z^1<0$ and $A^2_2 z^2 <0$, which implies that the SMFCQ holds at $(\bar{p}^k,\bar{x}^k)$ for each player $k$. Moreover, 
$$
     K(I_1,I_2) = \left\{y=(y^1,y^2) \, \big{|} \, y^1_1\le 0, y^2_1\le 0  \right\},
$$
and the system
$$
\left\{\begin{array}{ll}
    \frac{1}{2}(y^1_1)^{\mathrm{T}} P^1_{11}y^1_1 + (y^1_1)^{\mathrm{T}} P^1_{12}y^2_1  = \frac{1}{2}(y^1_1)^2 + y^1_1 y^2_1 \le 0,\\
    (y^2_1)^{\mathrm{T}} P^2_{21}y^1_1 + \frac{1}{2}(y^2_1)^{\mathrm{T}} P^2_{22}y^2_1  = \frac{1}{2}(y^2_1)^2 + 2y^1_1 y^2_1 \le 0
\end{array}\right.     
$$
has no nonzero solutions for $y\in K(I_1,I_2)$. Consequently, $S_{\rm KKT}$ is robust isolated calm at $\bar{\varepsilon} = 0$ for $(\bar{x},\bar{\lambda}) = (0,0,0,0)$, and for any $(\varepsilon,x(\varepsilon),\lambda(\varepsilon))\in {\rm gph}\, S_{\rm KKT}$ around $(0,\bar{x},\bar{\lambda})$, $x(\varepsilon)$ is an NE of Problem (\ref{NEP_perturbations}). In fact, direct computations yield that 
$$
    S_{\rm KKT}(\varepsilon) = \left\{\begin{array}{lll}\left\{ \left(2\e,-6 \varepsilon, 3\varepsilon, 0 \right)^{\mathrm{T}}\right\} \bigcup \left\{ (-2\varepsilon, \e , 0, \e)^{\mathrm{T}} \right\} \bigcup 
          \left\{ (-\varepsilon, 0 , 0, 0)^{\mathrm{T}} \right\} & & {\rm when}\ \varepsilon>0,\\
    \left\{(2\e,\e,-4\e,-7\e)^{\mathrm{T}}\right\} & & {\rm when}\ \varepsilon \le 0.
        
    \end{array}\right.
$$

\section{Conclusions}\label{Con}
In this paper, we analyze the stability properties of the KKT solution mapping $S_{\rm KKT}$ for a standard NEP with canonical perturbations. Firstly, we establish the exact characterizations of the strong regularity and the continuously differentiable single-valued localization of $S_{\rm KKT}$. Secondly, we propose {\rm the I-property} for the isolated calmness of $S_{\rm KKT}$ without convex assumptions, and {\rm the P-property} for the robust isolated calmness of $S_{\rm KKT}$ under convex assumptions. If Problem (\ref{NEP_perturbations}) is convex, {\rm the I-property} is less restrictive than {\rm the P-property}. At the end of each section, we illustrate the equivalence between the stability properties of $S_{\rm KKT}$ and $\T$, i.e., the KKT solution mapping for a standard NEP with only tilt perturbations. Finally, we provide detailed characterizations of stability for the NEP in which each individual player solves a QP problem.

Nevertheless, there are many other unresolved stability issues in NEPs. For instance, is {\rm the I-property} also necessary for the isolated calmness of $S_{\rm KKT}$? In the context of NLP, the answer is positive, where {\rm the I-property} actually reduces to the SOSC (\citet{dontchev2020characterizations}) under the SONC for optimality. Moreover, we obtain the robust isolated calmness of $S_{\rm KKT}$ from the convex assumptions and the P-property. Is it possible that we make {\rm the P-property} less restrictive, such as adopting {\rm the I-property}, or eliminate the convex assumptions to obtain the robust isolated calmness of $S_{\rm KKT}$? These are all challenging questions given that finding NEs is also a fixed point problem.


%
%
%

\ACKNOWLEDGMENT{The authors thank all reviewers for their valuable comments and suggestions. The authors also thank Jiani Wang for her valuable feedback.}



\bibliographystyle{informs2014} 
\bibliography{references} 





\end{document}